\documentclass[11pt]{article}
\usepackage{amsmath, graphicx, amsfonts,amssymb}
\usepackage{amsfonts,mathrsfs, color, amsthm}
\usepackage{bm,color}

\usepackage[bookmarks=true,
bookmarksnumbered=true, breaklinks=true,
pdfstartview=FitH, hyperfigures=false,
plainpages=false, naturalnames=true,
colorlinks=true,pagebackref=false,
pdfpagelabels]{hyperref}

\usepackage{hyperref}
\hypersetup{
  colorlinks   = true,
  urlcolor     = blue,
  linkcolor    = blue,
  citecolor   = red    }
\usepackage{enumerate}
\usepackage{mathtools}

\mathtoolsset{showonlyrefs}
\newtheorem{theorem}{Theorem}[section]
\newtheorem{definition}[theorem]{Definition}
\newtheorem{lemma}[theorem]{Lemma}
\newtheorem{corollary}[theorem]{Corollary}
\newtheorem{remark}[theorem]{Remark}
\newtheorem{proposition}[theorem]{Proposition}

\newtheorem{problem}[theorem]{Problem}

\numberwithin{equation}{section}
\def\Curs{\widetilde{C}_{\omega}^s}
\def\Cure{\widetilde{C}_{\omega}^e}
\def\DO{\widetilde{\delta}_{\omega}}
\def\LC{{\mathrm{LC}}_n}
\def\LO{{\mathrm{LC}}_n^o}
\def\TV{\mathrm{TV}}
\def\VO{\widetilde{V}_{\omega}}
\def\ball{B^n_2}
\def\sphere{S^{n-1}}
 \def\R{\mathbb{R}}

 \def\dom{{\mathrm {dom}}}

 \def\sphere{S^{n-1}}
 
 \def\LP{\mathscr{L}_e^{+}(\mu)}

\begin{document}

\title{The dual Orlicz curvature measures for log-concave functions and their related Minkowski problems }

\author{Niufa Fang,   Deping Ye,  Zengle Zhang and Yiming Zhao }
\date{}

\maketitle

\begin{abstract}  
The variation of a class of Orlicz moments with respect to the Asplund sum within the class of log-concave functions is demonstrated. Such a variational formula naturally leads to a family of dual Orlicz curvature measures for log-concave functions. They are functional analogs of dual (Orlicz) curvature measures for convex bodies. Partial existence results for the functional dual Orlicz Minkowski problem are shown.
\end{abstract}

\vskip 0.5cm
\section{Introduction}
\vskip 0.3cm

Let $\LC$ be the set of upper semi-continuous log-concave functions on $\R^n$ with nonzero finite $L^1$ norm and $\mathcal{K}^n$ be the set of convex bodies (compact convex sets with nonempty interiors) in $\mathbb{R}^n$. The set $\mathcal{K}^n$ can be naturally viewed as a subset of $\LC$ via the map $K\mapsto \mathbf{1}_K$ that takes a convex body $K$ to its indicator function $\mathbf{1}_K$. Recent years have witnessed extensions of many geometric results on $\mathcal{K}^n$ to their corresponding ones on $\LC$. Among those are the functional analogs of geometric isoperimetric (or reverse isoperimetric) inequalities: the Pr\'ekopa-Leindler inequality (functional Brunn-Minkowski) (see, for example, \cite[Theorem 7.1.2]{Gar02}),  the functional Santal\'o type inequality \cite{AKM04,Ball-thesis, FM08}, the affine Sobolev inequality (functional Petty projection inequality) \cite{HS2009,LYZ2002, Zhang1999}, the functional Zhang projection inequality \cite{ABM}, the functional differential affine isoperimetric inequality \cite{AKSW12}, and many more. The beautiful survey \cite{Col17} by Colesanti is a rich source of ``many surprising analogies between the theory of convex bodies and that of log-concave functions''. See also \cite{CLM2022, CLM2023} for recent efforts to establish the Hadwiger theorem on the set of convex functions. Parallel to the study of isoperimetric inequalities comparing different geometric invariants is the study of geometric measures obtained from properly differentiating geometric invariants of interest. This line of research goes back to the classical Minkowski problem \cite{CY1976,Minkowski1987, Minkowski1903, Pogorelov1978} and Aleksandrov's variational formula \cite{Aleksandrov1939} which justifies the heuristic idea that ``surface area measure is the derivative of volume''. The functional analog of Aleksandrov's variational formula came much later and took the combined effort of several works in the last decade \cite{CF13,CK15,Rot20,Rot22}. Surprisingly, as first revealed by Colesanti-Fragal\`a \cite{CF13}, different from its geometric counterpart, the corresponding variational formula for $L^1$ norm of a log-concave function with respect to the Asplund sum leads to two measures---one on $\R^n$ and one on $\sphere$. The characterization problem of the latter has a background in complex analysis and can be seen as the functional version of the classical Minkowski problem, which was completely solved within the class of essential continuous log-concave functions by  Cordero-Erausquin and Klartag in \cite{CK15}.

Roughly speaking, there are two variants of the classical Brunn-Minkowski theory in convex geometry. One is the $L_p$-Brunn Minkowski theory introduced by Lutwak \cite{Lut93}. Lutwak's groundbreaking work marks the start of a fruitful exploration (see, for example, \cite{BLYZ13,CHLL2020,CW06, HS20092}) in the last three decades. The $L_p$-Brunn-Minkowski theory also harbors one of the most sought-after conjectures in the field of convex geometric analysis, the log-Brunn-Minkowski conjecture. The log-Brunn-Minkowski conjecture, if proven, is a strengthening of the celebrated Brunn-Minkowski inequality and has been verified in dimension $2$ \cite{BLYZ12} as well as in many special cases in higher dimensions (see, for example, \cite{CHLL2020,vanhandel2002,KM2002, Milman2021,Stancu2003}). Functional analogs of the $L_p$ Minkowski problem were studied recently for $p\in (0,1)$ by Rotem \cite{Rot20} and for $p\in (1,\infty)$ by Fang-Xing-Ye \cite{FXY20+}.

Another variant of the Brunn-Minkowski theory is the dual Brunn-Minkowski theory, which can be credited for the eventual resolution of the Busemann-Petty problem \cite{Gardner1994,GKS1999, Lutwak1988, Zhang19992}. Dual quermassintegrals (averages of lower-dimensional sectional areas of a convex body) are the fundamental geometric invariants in the dual theory. In a pioneering work \cite{HLYZ16}, Huang-Lutwak-Yang-Zhang established the variational formula for dual quermassintegrals, which, for the first time, introduced geometric measures in the dual theory. They are known as dual curvature measures and their characterization problems are known as the dual Minkowski problems \cite{BHP2018, HLYZ16,LSW2020,Zhao2018}.  The functional analogs of dual quermassintegrals are nothing but moments (of different orders) of a log-concave function. Their variational formulas and the Minkowski problems they led to were studied in Huang-Liu-Xi-Zhao \cite{HLXZ}.

The Huang-Lutwak-Yang-Zhang approach to the variational formula is dual to Aleksandrov's, which does not require \emph{a priori} an isoperimetric inequality. This dual approach makes it possible to study variational formulas of ``volume'' functional in much greater generality. In particular, Gardner-Hug-Weil-Xing-Ye \cite{GHWXY19} and, subsequently, Gardner-Hug-Xing-Ye \cite{GHXY2020} demonstrated a family of variational formulas for a wide range of dual Orlicz volumes. Their works greatly expanded \cite{HLYZ16}. See, also, \cite{HLYZ2010,LYZ2010} (and the works citing them) for the Orlicz extension of the classical Brunn-Minkowski theory. 

It is worth pointing out that the dual (Orlicz) Brunn-Minkowski theory is not translation-invariant. Therefore in dealing with objects within this framework, we typically require some conditions regarding the relative position between the origin and the object under study.

The purpose of the current work is to demonstrate the variational formula for functional Orlicz moments within the class of log-concave functions under some mild assumptions. The other part of the current work studies the Minkowski problem for the dual Orlicz curvature measures generated by differentiating the Orlicz moments. It should be emphasized that it is often the case that functional results are more challenging to establish than their geometric counterparts, as convex bodies (via the natural identification) are only a very special type of log-concave functions.

Let $\mathscr{G}$ be the set of continuous functions $\omega: \R^n\setminus \{o\}\rightarrow (0, \infty)$ such that: 
\begin{itemize} 
\item[(A1):] $\omega$ is locally integrable in $\R^n$;

\item[(A2):] $\limsup_{|x|\rightarrow 0}|x|^n \omega(x)<\infty$, where $|x|$ denotes the Euclidean norm of $x\in\mathbb{R}^n$;
\item [(A3):] the following growth condition holds  \begin{align} \limsup_{|x|\rightarrow\infty} \frac{\ln \omega(x)} {|x|}=0. \label{condition-o-var}\end{align}  
\end{itemize} 
Note that since $\omega$ is continuous on $\R^n\setminus \{o\}$, Condition (A1) is equivalent to saying that $\omega$ is integrable on the centered unit ball. 

Let $\omega\in \mathscr{G}$ be fixed. For each $f\in \LC$, define its $\omega$-Orlicz moment by
\begin{align}\label{or-m-log}
\widetilde{V}_\omega(f)=\int_{\mathbb{R}^n}f(x)\omega(x)dx,
\end{align}
where the integration is with respect to the standard Lebesgue measure on $\R^n$.
As we shall see from Lemma \ref{Le.finiteness}, Conditions (A1) and (A3) ensure that $\widetilde{V}_\omega(f)$ is well-defined for each $f\in \LC$.

The notion of the Orlicz moment recovers many known quantities, both geometric and functional. 
In the functional setting, when $\omega\equiv1$, \eqref{or-m-log} simply becomes the $L^1$ norm of a log-concave function---an invariant that is the subject in the Pr\'{e}kopa-Leindler inequality, and the aforementioned variational formula that traces back to Colesanti-Fragal\`{a} \cite{CF13} and the work of Cordero-Erausquin and Klartag \cite{CK15}, and shown in full generality only recently by Rotem \cite{Rot22}. When $\omega(x)=|x|^{q-n}$ for some $q>0$, \eqref{or-m-log} becomes the $(q-n)$-th moment whose variational formula was studied in \cite{HLXZ}. To see that \eqref{or-m-log} is the natural extension of geometric invariants, one simply let $f=\mathbf{1}_K$ for some convex body $K\subset \mathbb{R}^n$ that contains the origin in its interior. In this case, when $\omega\equiv 1$, \eqref{or-m-log} becomes the volume of $K$. When $\omega(x)=|x|^{q-n}$ for some $q>0$, \eqref{or-m-log} becomes the dual quermassintegral of $K$. When $\omega$ is the Gaussian density, then \eqref{or-m-log} recovers the Gaussian volume of a convex body. For generic $\omega$, \eqref{or-m-log} recovers (special cases of) the general dual Orlicz volume $\widetilde{V}_\omega(\mathbf{1}_K)$ studied in \cite{GHWXY19, GHXY2020}. For simplicity, we will slightly abuse the notation and write $\widetilde{V}_\omega(K)$ instead of $\widetilde{V}_\omega(\mathbf{1}_K)$.

Using the dual variational approach pioneered by Huang-Lutwak-Yang-Zhang \cite{HLYZ16}, a version of \cite[Theorem 5.3]{GHWXY19} states: if $K\subset \mathbb{R}^n$ is a convex body that contains the origin in its interior and $L$ is a compact convex set, then   
\begin{align}
\label{eq 8.29.1}
\lim_{t\to 0^+}  \frac{\widetilde{V}_\omega(K+tL)-\widetilde{V}_\omega(K)}{t}=\int_{S^{n-1}}\frac{h_L(v)}{h_K(v)} d \widetilde{C}_\omega(K,v).
\end{align}   
Here $\widetilde{C}_\omega(K,\cdot)$  is a Borel measure on $\sphere$ given by 
\begin{align}\label{dual-curvature-convex body}
\widetilde{C}_\omega(K,\eta)=\int_{\{x\in\partial K:\nu_K(x)\in\eta\}}\langle x,\nu_K(x)\rangle \omega(x)d\mathcal{H}^{n-1}(x),
\end{align} 
for every Borel subset $\eta \subset \sphere$, where $\partial K$ denotes the boundary of $K$, $\langle \cdot, \cdot\rangle$ is the inner product, $\nu_K(\cdot)$ denotes the Gauss map of $K$, and $\mathcal{H}^{n-1}$ is the $(n-1)$-Hausdorff measure. 

Let $f=e^{-\phi}$ and $g=e^{-\psi}$ be two upper semi-continuous log-concave functions. For each $t>0$, one way to define the Asplund sum of $f$ and $g$ is via the formula
\begin{equation}
	(f\oplus t\cdot g)(z) = \sup_{x+t y=z} f(x)g(y)^t.
\end{equation}
The Asplund sum can be viewed as the functional version of the Minkowski addition. In particular, when $f=\mathbf{1}_K$ and $g=\mathbf{1}_L$ for some subsets $K,L\subset \R^n$, it is simple to see that $\mathbf{1}_K\oplus t\cdot \mathbf{1}_L=\mathbf{1}_{K+tL}$.

Equation \eqref{eq 8.29.1} shows that the $\omega$-Orlicz moment is differentiable in the subclass of characteristic functions with respect to the Asplund sum. It is natural to consider whether \eqref{eq 8.29.1} is valid for generic log-concave functions; that is, if $f\in \LC$ and $g$ is an upper-semi continuous log-concave function, does the following limit exist
\begin{equation}
	\lim_{t\to0^{+}}\frac{\VO(f\oplus t\cdot g)-\VO(f)}{t}
\end{equation}
 and if so, what is it? 
 
 Let $f=e^{-\phi} \in \LC$. Define its \emph{Euclidean dual $\omega$-Orlicz curvature measure} $\widetilde{C}_{\omega}^{e}(f,\cdot)$ to be the Borel measure on $\R^n$ given by
 \begin{equation}
 	\Cure(f, E)=\int_{\{x\in \R^n:\ \nabla\phi(x)\in E\}}e^{-\phi(x)}\omega(x)\,dx,
 \end{equation}
 for each Borel set $E\subset \R^n$, where $\nabla \phi$ is the gradient of $\phi$. In short, $\Cure(f, \cdot)$ is the push-forward of the measure $e^{-\phi(x)}\omega(x)dx$ by $\nabla \phi$. Note that since $\phi$ is convex, it is almost everywhere differentiable. Define the \emph{spherical dual $\omega$-Orlicz curvature measure} $\Curs(f, \cdot)$ to be the Borel measure on $\sphere$ given by
 \begin{equation}
 	\Curs(f, \eta)=\int_{\{x\in \partial K_f:\  \nu_{K_f}(x)\in \eta\}} e^{-\phi(x)}\omega(x)\,d\mathcal{H}^{n-1}(x),
 \end{equation}
 where $K_f =\overline{\{x\in \R^n: f(x)\neq 0\}}$ is the closure of the support of $f$ and $\nu_{K_f}$ is its Gauss map. Since $f$ is log-concave, the Gauss map $\nu_{K_f}$ is defined almost everywhere on $\partial K_f$ with respect to the $(n-1)$-dimensional Hausdorff measure $\mathcal{H}^{n-1}$.
 
 We will establish the following result.
 \begin{theorem}\label{thm 8.29.1}
 Let $\omega\in \mathscr{G}$. Assume $f=e^{-\phi}\in \LC$ is a log-concave function that achieves its maximum at the origin and 
\begin{align}\label{eq 8.29.3}
\limsup_{|x|\to 0} \frac{|f(x)-f(o)|}{|x|^{\alpha+1}}<\infty,
\end{align} 
for some $0<\alpha<1.$ Then, if  $g=e^{-\psi}$ is an upper semi-continuous, log-concave function with compact support and $g(o)>0$, we have
\begin{align}
\lim_{t\to0^{+}}\frac{\VO(f\oplus t\cdot g)-\VO(f)}{t}=\int_{\mathbb{R}^n}\psi^\ast(y)d\Cure(f;y)+
\int_{S^{n-1}}h_{K_g}(v)d\Curs(f;v).
\end{align}
\end{theorem}
Here, $h_{K_g}$ is the support function of $K_g$ (see \eqref{eq support function}) and $\psi^*$ is the Legendre transform of $\psi$ (see \eqref{def-dual}). Note that Condition (A2) on $\omega$ is used critically in the proof of Theorem \ref{thm 8.29.1}.

When $f$ and $g$ are characteristic functions, Theorem \ref{thm 8.29.1} recovers \eqref{eq 8.29.1}. We point out that when $\omega\equiv1$, Theorem \ref{thm 8.29.1} has been established in \cite{Rot22} without the extra condition \eqref{eq 8.29.3}. The special case $\omega\equiv 1$, however, does have a unique feature that makes the removal of \eqref{eq 8.29.3} immediate: the integral of a function remains invariant if the function $f$ is replaced by $f(\,\cdot\, +x_0)$. For a generic Orlicz moment $\widetilde{V}_{\omega}$, it is \emph{not} translation invariant. This phenomenon can also be observed in the geometric setting: note that \eqref{eq 8.29.1} only holds when $K$ contains the origin as an interior point, a restriction not required in the volume variational formula by Aleksandrov. We remark here that intuitively, the extra condition \eqref{eq 8.29.3} is used to ensure that the level sets of $f$ contain the origin as an interior point in some uniform fashion. If $f$ is $C^{1,\alpha}$ in a neighborhood of the origin, then condition \eqref{eq 8.29.3} is satisfied. Theorem \ref{thm 8.29.1} recovers the special case $\omega(x)=|x|^{q-n}$ obtained in \cite{HLXZ}. It is of great interest to see whether condition \eqref{eq 8.29.3} can be replaced by the mere assumption that the support of $f$ contains the origin as an interior point.

Minkowski problems are an important family of problems that can be posed for measures obtained through differentiating important invariants. The study of Minkowski problems is often intertwined with that of PDEs and thus attracts attention from those working in the field of fully nonlinear elliptic PDEs as well as those in the field of convex geometric analysis.  It is worth pointing out that the development of Monge-Amp\`{e}re equations in the previous century is largely parallel to that of the classical Minkowski problem, see, for example, \cite{Caffarelli1989, Caffarelli1990,CY1976, Pogorelov1978}. We also would like to point out \cite{BF2019, CHZ2019, CL2021,CL, LSW2020,LW13} for variants of the classical Minkowski problem (such as the $L_p$ Minkowski problem, the dual Minkowski problem, and \emph{etc}). The readers are referred to these works and the works citing them to see how Minkowski problems connect geometry with PDEs, and also how they connect with isoperimetric inequalities. 

For the second part of the current work, we consider the following Minkowski-type problem in the setting of log-concave functions. 

\noindent\textbf{The functional dual Orlicz Minkowski problem.} {\em Given a finite Borel measure $\mu$ on $\R^n$, find necessary and sufficient conditions on $\mu$ so that there exists $f\in \LC$ such that $\mu=\Cure(f, \cdot)$?}

When $\mu$ has a $C^{\infty}$ density $g\geq 0$, then the functional dual Orlicz Minkowski problem reduces to the following PDE on $\R^n$:
\begin{equation}
	g(\nabla \phi(x))\det(\nabla^2 \phi(x))=\omega(x)e^{-\phi(x)}.
\end{equation}
Here, one is looking for an unknown convex $\phi$ which recovers $f$ via the formula $f=e^{-\phi}$.

For each nonzero even finite Borel measure $\mu$ on $\R^n$, let $|\mu|=\int_{\mathbb{R}^n}d\mu$ and
\begin{equation}\label{c0}
	\zeta_{\mu} = \inf_{v\in \sphere} \int_{\R^n} |\langle x, v\rangle|d\mu(x).
\end{equation} It is simple to see that if $\mu$ is not concentrated in any proper subspace of $\R^n$, then $\zeta_{\mu}>0$. Moreover, if $\mu$ has finite first moment, \emph{i.e.}, $\int_{\mathbb{R}^n}|x|\,d\mu(x)<\infty,$ then $\zeta_{\mu}<\infty$. 

Let $\omega\in \mathscr{G}$ be fixed. Define $\mathcal{F}_\omega(t)=\VO\big(e^{-t|x|}\big)$ for $t>0$. In Section \ref{Se.dualMin}, we will show

\begin{theorem}
\label{thm 8.29.3}
	Let $\omega\in\mathscr{G}$ be an even function and $\mu$ be a nonzero even finite Borel measure on $\R^n$. If $\mu$ is not concentrated in any proper subspace, $\mu$ has finite first moment, and 
\begin{equation}\label{eq 8.29.4}
\liminf_{t\to 0^+ } \Big(\mathcal{F}_\omega(t) e^{-\frac{\zeta_{\mu}}{2|\mu|t}}\Big)=0,
\end{equation} 
then there exists $f\in \LC$ such that   
\begin{equation}\label{eq 8.29.7}
\mu= \widetilde{C}^{e}_{\omega}(f,\cdot).
\end{equation}
\end{theorem}

We remark here that for each fixed $\omega\in \mathscr{G}$, \eqref{eq 8.29.4} is a condition on $\mu$ to guarantee the existence of a solution. However, for many choices of $\omega$, it turns out that \eqref{eq 8.29.4} is satisfied for all $\mu$. These choices of $\omega$ include in particular any $\omega$ with at most polynomial growth (in particular, if $\omega$ is bounded). See Remark \ref{remark 8.29.1} for additional details. Thus, Theorem \ref{thm 8.29.3} recovers \cite[Theorem 2]{CK15} in the origin-symmetric case (without the conclusion that $f$ is essentially continuous) and \cite[Theorem 1.4]{HLXZ}. It is important to notice that in \cite{CK15, HLXZ}, a crucial tool to establish the existence is a Santal\'o type inequality. A major hurdle in the functional dual Orlicz Minkowski problem is the lack of such inequality in the Orlicz case. Lemma \ref{5.8} offers a way to go around the Santal\'o type inequalities. It is worth pointing out that the assumption that $\mu$ has finite first moment is necessary, due to Proposition \ref{Le.finite}.

Readers familiar with Orlicz Minkowski problems might wonder how \eqref{eq 8.29.7} can hold without an additional constant $\tau>0$ that typically appears in the solutions of an Orlicz Minkowski problem for convex bodies. It is important to note that in Orlicz Minkowski problems for convex bodies, one is looking for a convex body as a solution, or, equivalently, the indicator function $\mathbf{1}_K$. In the functional setting, one is looking for a log-concave function $f$ as a solution. One may rescale $f$ to be $\tau f$. However, $\tau \mathbf{1}_K$ is no longer an indicator function.

\section{Preliminaries}\label{Section-2}

In this section, we provide notations and basic results necessary for carrying out the current work. Since our results (on log-concave functions) depend heavily on the theory of convex bodies, we include basics regarding convex bodies first. The comprehensive book \cite{Sch14} by Schneider is a highly recommended reference if the readers wish to learn more. The second half of this section includes basics regarding log-concave (or, equivalently, convex) functions, as well as the theory of functions of bounded variations. The books \cite{AFP, EG2015} should be consulted for basics on BV functions and the book \cite{Roc70} is an excellent reference for log-concave functions.

Throughout the current paper, we work exclusively in the $n$-dimensional Euclidean space $\R^n$ with some fixed $n\geq 1$.  We write $|x|$ for the standard Euclidean norm and $\langle x, y\rangle$ for the standard inner product of $x,y\in \R^n$. We denote by $\ball(x,r)$ the Euclidean ball of radius $r$ centered at $x$. When $x=o$, we write $\ball(r)$ for simplicity. In addition, we shall write $\ball$ for the unit ball centered at $o$. By $V_n(\cdot)$, we mean the volume functional in $\mathbb{R}^n$. Let $E$ be a subset of $\R^n$. Its boundary, closure, interior, and complement will be denoted by $\partial E$, $\overline{E}$, $\mathrm{int}(E)$, and $E^c$, respectively. The $k$-dimensional Hausdorff measure restricted to $E$ (if exists) will be written as $\mathcal{H}^k|_E$ or simply abbreviated as $\mathcal{H}^k$ if there is no confusion. In integrals over the unit sphere, we will write $du$ instead of $d\mathcal{H}^{n-1}(u)$. We will use $\nu_K$ for the Gauss map of a closed, convex set $K$. Note that since $K$ is convex, its Gauss map $\nu_K$ is defined almost everywhere on $\partial K$ with respect to $\mathcal{H}^{n-1}$.

Let $\mathcal{K}^n$ be the set of convex bodies (compact, convex, with nonempty interiors) in $\R^n$. The set $\mathcal{K}^n_o$ denotes the subclass consisting of only those convex bodies containing the origin as an interior point. 

The support function $h_K:S^{n-1}\rightarrow \mathbb{R}$ of a compact, convex set $K\subset \R^n$ is defined by \begin{align} 
\label{eq support function}
h_K(v)=\max\{\langle v,y\rangle: y\in K\}.
\end{align} 
Occasionally, we will abuse the notation and treat $h_K$ as a function on $\R^n$ by extending it to be a 1-homogeneous function. For two compact convex subsets $K, L\subset \R^n$ and $a, b>0$, their Minkowski combination $aK+bL$ can be defined as 
\begin{equation}
	aK+bL = \{ax+by:x\in K, y\in L\}.
\end{equation}
Equivalently, one may define $aK+bL$ as the unique compact convex set whose support function is $ah_K+bh_L$.

Let $g\in C(\sphere)$, that is, $g$ is a continuous function on $\sphere$. If $g>0$, its Wulff shape $[g]$ is defined to be the convex body given by 
\begin{equation}
	[g]=\{x\in \R^n: \langle x, v\rangle\le g(v),\ \text{for all}\ v\in \sphere\}.
\end{equation}
Clearly, $[g]\in \mathcal{K}^n_o$. Also obvious is the fact that $h_{[g]}\leq g$. 

Another function associated with $K\in \mathcal{K}_{o}^n$ is the radial function: $\rho_K:\sphere \to(0,\infty)$ given by 
\begin{equation}
	\rho_K(u)=\max\{\lambda>0: \lambda u\in K\}.
\end{equation}
One may view $\rho_K$ as a function on $\mathbb{R}^n\setminus \{o\}$ by extending it to be homogeneous of degree $-1$. The radial function and the support function can be related via the polar body operation. The polar body of $K\in \mathcal{K}^n_o$,  denoted by $K^\circ$, is defined by
\begin{equation}
	K^{\circ}=\{y\in \R^n: \langle x, y\rangle \leq 1 \ \text{for all} \ x\in K\}.
\end{equation}
It is not hard to see that $h_K = \rho_{K^\circ}^{-1}$.

The Minkowski functional $\|\cdot \|_{K}$ of $K\in \mathcal{K}_o^n$ can be defined by $\|x\|_K = h_{K^\circ}(x)$. It is not hard to see that $\|\cdot\|_K$ is a (potentially asymmetric) norm on $\R^n$ and that $K$ is the unit ball centered at $o$ with respect to this norm. This offers another way to view the polar body operation: $K$ and $K^\circ$ are the centered unit balls of the Banach space $(\R^n, \|\cdot\|_K)$ and its dual space. 

Let $K\in \mathcal{K}_o^n$ and $g:\sphere \rightarrow \R$ be a continuous function. For $t\in \R$, define $h_t = h_K+tg$. Since $h_K>0$ and $h_K,g$ are continuous on $\sphere$, there exists $\delta_0>0$ such that for each $t\in (-\delta_0, \delta_0)$, the function $h_t$ is a positive continuous function on $\sphere$. Define 
\begin{equation}
\label{eq 9.5.1}
	K_t=[h_t], \text{ for all } t\in (-\delta_0, \delta_0).
\end{equation} 
An application of \cite[Lemma 4.3]{HLYZ16} implies that, for almost all $u\in \sphere$ with respect to $\mathcal{H}^{n-1}$,  
\begin{equation}
\label{eq 9.5.5}
	\left.\frac{d}{dt}\right|_{t=0} \rho_{K_t}(u) =  \rho_K(u)\frac{g\left(\nu_K(u\rho_K(u))\right)}{h_K(\nu_K(u\rho_K(u)))}.
\end{equation} 
Moreover, \cite[Lemma 4.2 and Lemma 2.8]{HLYZ16} imply the existence of $M>0$ such that 
\begin{equation}
\label{eq 9.5.4}
	\left|{\rho_{K_t}(u)-\rho_{K}(u)}\right|\leq M|t|,
\end{equation}
for almost all $u\in \sphere$ and all $t\in (-\delta_0,\delta_0)$. 

Let $\omega: \mathbb{R}^n\setminus \{o\}\rightarrow (0,\infty)$ be a continuous function satisfying Condition (A1). For each $K\in \mathcal{K}_o^n$, define its dual $\omega$-Orlicz volume by
\begin{equation}
	\label{GDVofK}
\widetilde{V}_\omega(K)  =\int_{K}\omega(x)dx.
\end{equation}
It is trivial to see that $\widetilde{V}_\omega(K)$ is well-defined since $\omega$ is locally integrable and $K\in \mathcal{K}_o^n$. Some special cases for  $\widetilde{V}_\omega(K)$  are the volume (when $\omega\equiv 1$), dual quermassintegrals (when  $\omega(x)=|x|^{q-n}$ for $q>0$), and the Gaussian measure of $K$ (when $\omega$ is the Gauss density). In particular, the dual Orlicz volume can be viewed as a weighted volume. It is worth mentioning that the introduction of  $\widetilde{V}_\omega(K)$ played important roles in the development of general dual Orlicz Brunn-Minkowski theory, see, for example, \cite{GHWXY19,GHXY2020,LW13, XY17}. 

For simplicity, we define $\overline{\omega}(t,u):[0,\infty)\times S^{n-1}$ by 
\begin{align}\label{overline-omega}
\overline{\omega}(t,u)
=\int_{0}^t\omega(ru)r^{n-1}dr.
\end{align}  
Since $\omega$ satisfies Condition (A1), by polar coordinates and the Fubini theorem, it is simple to see that $\overline{\omega}(t,u)$ is well-defined for almost all $u\in \sphere$ with respect to $\mathcal{H}^{n-1}$ and all $t\in [0,\infty)$. In particular, if $u\in \sphere$ is such that $\overline{\omega}(\cdot, u)$ is defined on $[0,\infty)$, then by the fundamental theorem of calculus, one has, for every $t>0$,
\begin{equation}
\label{eq 9.5.3}
	\frac{\partial}{\partial t} \overline{\omega}(t,u) = \omega(tu)t^{n-1}.
\end{equation}

We recall the following version of the variational formula established in \cite[Theorem 5.3]{GHWXY19}. Note that the restriction on $\omega$ is slightly loosened when compared to that in \cite{GHWXY19}. We therefore provide a short proof for completeness.

\begin{lemma}\label{Le.var.geo}
Let $\omega:\R^n\setminus \{o\}\rightarrow (0,\infty)$ be a continuous function, $K\in \mathcal{K}^n_o$, and $g\in C(\sphere)$. Assume $\omega$ satisfies Condition (A1) and define $K_t=[h_t]=[h_K+tg]$. Then 
\begin{align}\label{Eq.V1,phi(K,L)}
\lim_{t\to 0}\frac{\widetilde{V}_\omega(K_t)-\widetilde{V}_\omega(K)}{t}
=\int_{\sphere}\frac{g(v)}{h_K(v)}d\widetilde{C}_\omega(K,v),\end{align}
where $\widetilde{C}_\omega(K,\cdot)$, a finite Borel measure on $\sphere$ defined as in \eqref{dual-curvature-convex body}, is the dual Orlicz curvature measure of $K$. 
\end{lemma}
\begin{proof}
	Let $u\in \sphere$ be such that $\overline{\omega}(t,u)$ is defined for all $t\in [0,\infty)$. Such $u$ is almost everywhere on $\sphere$. 
	Moreover, for sufficiently small $|t|$, since $\rho_{K_t}>0$, by the mean value theorem and \eqref{eq 9.5.3}, there exists $\theta$ between $\rho_{K_t}(u)$ and $\rho_{K}(u)$ such that 
	\begin{equation}
	\label{eq 9.5.7}
		\frac{1}{t}\int_{\rho_K(u)}^{\rho_{K_t}(u)} \omega(ru)r^{n-1}dr= \frac{\overline{\omega}(\rho_{K_t},u)-\overline{\omega}(\rho_{K},u)}{t}=\omega(\theta u)\theta^{n-1} \frac{\rho_{K_t}(u)-\rho_K(u)}{t}. 
	\end{equation} This, when combined with \eqref{eq 9.5.4} and that $\theta u$ is in a compact subset of $\mathbb{R}^n\setminus \{o\}$ (implied by $K\in \mathcal{K}^n_o$ and that $|t|$ is sufficiently small), implies that there exists $M>0$ independent of $t$ and $u$ such that 
	\begin{equation*}
		\left|\int_{\rho_K(u)}^{\rho_{K_t}(u)} \omega(ru)r^{n-1}dr\right|<M|t|.
	\end{equation*}
	
	By the definition of dual Orlicz volume, polar coordinates, the bounded convergence theorem, \eqref{eq 9.5.7}, \eqref{eq 9.5.5}, the fact that $\theta$ is between $\rho_{K_t}(u)$ and $\rho_K(u)$, the continuity of $\omega$ on $\mathbb{R}^n\setminus \{o\}$, and the definition of $\widetilde{C}_{\omega}(K,\cdot)$, we have
	\begin{equation}
		\begin{aligned}
			\lim_{t\rightarrow 0} \frac{\widetilde{V}_{\omega}(K_t)-\widetilde{V}_{\omega}(K)}{t}&= \lim_{t\rightarrow 0} \frac{1}{t} \left(\int_{K_t}\omega(x)dx-\int_{K}\omega(x)dx\right)\\
			&= \int_{\sphere} \left(\lim_{t\rightarrow 0}\frac{1}{t}\int_{\rho_{K}(u)}^{\rho_{K_t}(u)} \omega(ru)r^{n-1}dr\right)du\\
			&= \int_{\sphere} \left(\lim_{t\rightarrow 0} \omega(\theta u)\theta^{n-1}\frac{\rho_{K_t}(u)-\rho_K(u)}{t}\right)du\\
			& = \int_{\sphere} \omega(\rho_K(u)u)\rho_K^{n}(u)\frac{g\left(\nu_K(u\rho_K(u))\right)}{h_K(\nu_K(u\rho_K(u)))}du\\
			& = \int_{\sphere} \frac{g(v)}{h_K(v)} d\widetilde{C}_{\omega}(K,v).
		\end{aligned}
	\end{equation}
\end{proof}

When $g=h_L$ is the support function of a compact convex set $L$, we will write the right-hand side of \eqref{Eq.V1,phi(K,L)} as $\widetilde{V}_{1,\omega}(K,L)$. That is
\begin{equation}
\label{eq 9.12.1}
	\widetilde{V}_{1,\omega}(K,L)= \int_{\sphere}\frac{h_L(v)}{h_K(v)}d\widetilde{C}_\omega(K,v).
\end{equation} 

 We now discuss basics regarding log-concave, or, equivalently, convex functions. The primary object we are studying is a log-concave function $f=e^{-\phi}: \R^n\rightarrow [0,\infty)$ that is upper semi-continuous and $L^1$ integrable. Here $\phi: \R^n\rightarrow \mathbb{R}\cup\{+\infty\}$ is a lower semi-continuous convex function.
 
 For simplicity, we will write $\LC$ for the collection of upper semi-continuous log-concave functions with nonzero finite $L^1$ norm.  It is well-known that a log-concave function $f=e^{-\phi}$ is integrable if and only if
 \begin{equation}
 \label{In.CF13}
 	\liminf_{|x|\rightarrow \infty} \frac{\phi(x)}{|x|}>0.
 \end{equation}
 See, for example, \cite[Lemma 2.5]{CF13}. 
 
 We will write $\LO$ for the subclass of $\LC$ containing only those functions that achieve maxima at the origin. Such log-concave functions are sometimes called \emph{geometric}. See, for example,  \cite{AM11}.
 
 Playing the role of polar body operation in the functional setting is the Legendre transform. See, \cite{AM09, BS2008}. For a function $\phi:\R^n\to\R\cup \{+\infty\}$ (not necessarily convex), its  {Legendre transform}, denoted by $\phi^*$, defines a lower semi-continuous convex function: for $y\in\mathbb{R}^n$,
\begin{equation}\label{def-dual}
\phi^*(y)=\sup_{x\in\mathbb{R}^n}\left\{\langle x,y\rangle-\phi(x)\right\}.
\end{equation}
Directly following from the definition is the fact that  \begin{align} 
\label{equ-dual-1} \phi(x)+ \phi^*(y)\geq \langle x,y\rangle \ \ \ \mathrm{for\ all\ }x, y\in\mathbb{R}^n. 
\end{align}
A simple consequence of this fact is that if $\phi\not\equiv +\infty$, then $\phi^*>-\infty$. By direct computation, if $a\in \mathbb{R}$, then $(\phi+a)^* = \phi^*-a$. Moreover, for $c>0$, we have 
\begin{equation}
\label{Eq.(avarphi)*}
	(c\phi)^* = c \phi^*(\cdot/c).
\end{equation}

For a subset $E\subset \mathbb{R}^n$, we write $\mathbf{1}_E:\mathbb{R}^n\rightarrow \mathbb{R}$ for its indicator function. That is, $\mathbf{1}_E(x)=1$ if $x\in E$ and $\mathbf{1}_E(x)=0$ if $x\not\in E$. The characteristic function of $E$ is denoted by $\chi_{E}$, namely, $\chi_E(x)=0$ if $x\in E$ and $\chi_E(x)=+\infty$ if $x\notin E$. If $K\in \mathcal{K}^n$, then 
\begin{align}
    (\chi_K)^*=h_K. \label{relation-supp-dual}  
\end{align}

We recall the following facts regarding Legendre transform:
\begin{enumerate}[(a)]
	\item the Legendre transform reverses order: if $\phi\geq \psi$, then $\phi^*\leq \psi^*$;
	\item for any function $\phi:\mathbb{R}^n\rightarrow \mathbb{R}\cup\{ +\infty\}$, we have $(\phi^*)^*=\phi^{**}\leq\phi$, with equality if and only if $\phi$ is a lower semi-continuous convex function.
\end{enumerate}

What replaces the Minkowski combination between convex bodies in the functional setting is the Asplund sum and the following scalar multiplication on the set of log-concave functions. Let $f=e^{-\phi}, g=e^{-\psi}$ be log-concave functions and $t>0$. Define 
\begin{equation}
	(f\oplus g) = e^{-(\phi^*+\psi^*)^*} \ \ \ \mathrm{and} \ \ 
	t\cdot f = e^{-(t\phi^*)^*}.
\end{equation}
Equivalently, the above definitions can be written as supremum convolution.  For log-concave functions $f,g$ and $t,s>0$, we have
\begin{equation}
	(t\cdot f\oplus s\cdot g)(z)=\sup_{tx+sy=z} f(x)^tg(y)^s.
\end{equation}
In the special case where both $f$ and $g$ are indicator functions of convex bodies, the combination $t\cdot f\oplus s\cdot g$ recovers the classical Minkowski combination between convex bodies. 

Let $\phi: \mathbb{R}^n\rightarrow \mathbb{R}^n \cap \{+\infty\}$ be a convex function. We use $\dom (\phi)$ for the effective domain of $\phi$. That is, $\dom (\phi)=\{x\in\mathbb{R}^n\,:\, \phi(x)<+\infty\}$. It is well-known that $\phi$ is almost everywhere differentiable in $\text{int}\, \dom(\phi)$ with respect to Lebesgue measure. Moreover, if $\nabla \phi(x)$ exists, then
\begin{equation}
	\label{basicequ}
	\phi(x)+\phi^*(\nabla\phi(x)) = \langle x, \nabla \phi(x)\rangle.
\end{equation}

We will write $K_f$ for the closure of the support of a function $f$. That is, $K_f = \overline{\{x\in \R^n: f(x)\neq 0\}}$. If $f$ is a log-concave function, then $K_f$ is convex and consequently $\nu_{K_f}$ is defined almost everywhere on $\partial K_f$ with respect to $\mathcal{H}^{n-1}$. Moreover, $f$ is almost everywhere differentiable on $\R^n$ with respect to the Lebesgue measure.

If $f:\R^n\rightarrow \mathbb{R}$ is an integrable function, its total variation is given by 
\begin{align}\label{TL(f)}
\TV(f)=\sup\left\{\int_{\mathbb{R}^n}f(x) \operatorname{div} \Phi(x) dx: \Phi\in C_c^1( \mathbb{R}^n, \R^n) \ \ \mathrm{and}\ \  |\Phi(x)|\le 1\ \ \mathrm{for\ any\ }x\in \R^n\right\}.
\end{align}
Here, $C_c^1( \mathbb{R}^n, \R^n)$ denotes the set of all compactly supported $C^1$ maps from $\R^n$ to $\R^n$.
 When $\TV(f)<\infty$, we say that $f$ is of bounded variation.  A well-known result for functions of bounded variation is the existence of a vector Radon measure $Df$ on $\mathbb{R}^n$ \cite{AFP,EG2015}, such that,
\begin{align}\label{RadomMeasure}
\int_{\mathbb{R}^n}f(x) \operatorname{div} \Phi(x) \,dx=-\int_{\mathbb{R}^n}\langle\Phi(x), \,d\,Df(x)\rangle.
\end{align} 
Moreover, one has $\,d\,Df=\sigma_f \,d\|Df\|$ for some finite Radon measure $\|Df\|$ on $\R^n$ and some  $\|Df\|$-measurable $n$-dimensional vector valued function  $\sigma_f: \R^n\rightarrow \R^n$  satisfying that  $|\sigma_f(x) | = 1$ for $\|Df\|$-almost all $x\in \R^n$. The measure $\|Df\|$ is called the variation measure of $f$.

Let $L\in \mathcal{K}_o^n$ and write the anisotropic variational measure of $f$ with respect to $L$ by 
\begin{equation*}
	d\|-Df\|_{L^*}=h_L(-\sigma_f)d\|Df\|.
\end{equation*}
Let $\omega\in \mathscr{G}$.
We say an integrable function $f$ is of bounded anisotropic weighted variation with respect to $(L, \omega)$ if $f$ is of bounded variation and $\omega$ is integrable with respect to $\|-Df\|_{L^*}$. In this case, we write
\begin{equation}
\label{eq 81620}
	\TV_{L,\omega}(f)=\int_{\R^n} \omega d\|-Df\|_{L^*}
\end{equation}  
From the proof of  \cite[Theorem 3.2]{Rot20}, one sees that, if $f\in \LC$, then $f$ is of bounded variation and
\begin{align}
\label{eq yz1}
Df
=\nabla f dx-
\nu_{K_f} f d\mathcal{H}^{n-1}(x)\Big|_{\partial K_f}.
\end{align} 
It will be shown in Proposition \ref{Le.finite} that if $f=e^{-\phi}\in \LC$, $o\in \text{int}\, \dom(\phi)$, and $\omega\in \mathscr{G}$, then $\TV_{L,\omega}(f)<\infty$. In particular, 
\begin{equation}
	d\|-Df\|_{L^*} = h_L(\nabla \phi) f dx+h_L(\nu_{K_f})fd\mathcal{H}^{n-1}\Big|_{\partial K_f}.
\end{equation}

Let $E\subset \R^n$. We say $E$ is of finite perimeter if $\mathbf{1}_E$ is of bounded variation and write
\begin{equation}
	\operatorname{Per} (\partial E) = \operatorname{TV}(\mathbf{1}_E).
\end{equation}
Similarly, whenever it exists, we will write
\begin{equation}
	\operatorname{Per}_{L,\omega} (\partial E) = \operatorname{TV}_{L,\omega}(\mathbf{1}_E).
\end{equation}

For $s\in \mathbb{R}$, we will use $E_s(f)$ for the superlevel set of $f$. That is
\begin{equation}
	E_s(f)=\left\{x\in\mathbb{R}^n: f(x)\ge s \right\}.
\end{equation}
When the context is clear, we will occasionally write $E_s$.

The following is a special case of the coarea formula established in  \cite[(2.22)]{FMP10}: if $L\in \mathcal{K}^n_o$,  $f=e^{-\phi}\in \LC$ with $o\in \text{int}\, \dom(\phi)$, and  $\omega\in \mathscr{G}$, then
\begin{align}\label{coarea-formula}
\TV_{L,\omega}(f)
=\int_{0}^{\infty}\operatorname{Per}_{L,\omega}(\partial E_s)ds.
\end{align}

\section{The Orlicz moment of log-concave functions}\label{Sec.variaition}

Recall that $\LC$ is the collection of upper semi-continuous log-concave functions on $\R^n$ with non-zero finite $L^1$ norm. Let $\omega\in \mathscr{G}$. We define the $\omega$-Orlicz moment of $f$ as in \eqref{or-m-log}:   
\begin{align*}
\widetilde{V}_\omega(f)=\int_{\mathbb{R}^n}f(x)\omega(x)dx.
\end{align*}
As will be seen in Lemma \ref{Le.finiteness}, as long as $\omega$ satisfies (A1) and (A3), the moment $\widetilde{V}_{\omega}(f)$ will always be finite. When $f=\mathbf{1}_K$ for $K\in \mathcal{K}_o^n$, it is simple to see that $\widetilde{V}_{\omega}(f)$ recovers the dual Orlicz volume $\widetilde{V}_{\omega}(K)$ of $K$. Equally simple to see is the fact that for every $q>0$, the function $\omega(x)=|x|^{q-n}\in \mathscr{G}$. Therefore, the $\omega$-Orlicz moment recovers the classical $q$-th moment of $f$ in probability theory. 

Immediate from the definition of $\widetilde{V}_{\omega}(\cdot)$ is its monotonicity: if $f\leq g$, then $\widetilde{V}_{\omega}(f)\leq \widetilde{V}_{\omega}(g)$. 

\begin{lemma}\label{Le.finiteness}  
For any $f=e^{-\phi}\in\LC$ and $\omega\in \mathscr{G}$, one has $0<\VO(f)<\infty$.  
\end{lemma}
\begin{proof}
That $\widetilde{V}_{\omega}(f)>0$ is immediate: if it was $0$, then $f(x)\omega(x)$ would be 0 almost everywhere, which, when combined with the fact that $\omega>0$, would imply that $f$ is almost everywhere 0. This would be a contradiction to the fact that $f$ has a non-zero $L^1$ norm.

We now show that $\widetilde{V}_{\omega}(f)<\infty$. Intuitively, the finiteness of $\widetilde{V}_\omega(f)$ comes from: 1. upper semi-continuity of $f$ and Condition (A1), which guarantees integrability around the origin; 2. the integrability of $f$ (hence the fast decay of $\phi$) and Condition (A3), which guarantees integrability near $\infty$. 
	
	Since $f\in \LC$, it follows from  \eqref{In.CF13} that $\phi(x)\geq c|x|+d$ for some constants $d\in \R$ and $c>0$. By Condition (A3), there exists a sufficiently large $M>0$ such that for all $x\in (B_2^n(M))^c$, we have
\begin{align}\label{o-v-estimamte-1}
\ln \omega(x)<\frac{c}{2}\cdot |x|+\frac{d}{2}\leq \frac{\phi(x)}{2}.\end{align}
Therefore,
\begin{equation}
\label{eq 8111}
	\begin{aligned}
		\int_{(\ball(M))^c} e^{-\phi}\omega\,dx   =\int_{(\ball( M))^c} e^{\ln \omega-\phi}\,dx  \leq  \int_{\R^n}  e^{\frac{-\phi}{2}}\,dx \leq \int_{\R^n}e^{-\frac{c|x|}{2}-\frac{d}{2}}dx<\infty.
	\end{aligned}
\end{equation}

Since $f$ is upper semi-continuous, there exists $C>0$ such that on $\ball(M)$, we have $0\leq f\leq C.$
Therefore, by Condition (A1), we have
\begin{equation}
\label{bound-at-o}
\begin{aligned}
\int_{\ball(M)}f\omega\,dx  \leq C \int_{B_2^n(M)} {\omega}(x)\,dx <\infty.
\end{aligned} 
\end{equation}
The desired result follows immediately from \eqref{eq 8111} and \eqref{bound-at-o}.
\end{proof}

We point out that many usual functions are in $\mathscr{G}$, such as, $|x|^{q-n}$ for $q>0$, all bounded continuous functions on $\R^n$, and $e^{|x|^{\alpha}}$ for some constant $0<\alpha<1$. Generally, if a continuous function $\omega: \R^n\setminus\{o\} \rightarrow (0, \infty)$ is bounded from above by $|x|^{q-n}$ for $q>0$ around the origin and bounded from above by $\left(e^{|x|^{\alpha}}\right)^{\beta}$ for some constants $\alpha<1$, $\beta\in \mathbb{R}$ around $\infty$, then $\omega\in \mathscr{G}$ as well. 

We have the following result. In particular, \eqref{In.finite} extends \cite[Proposition 7]{CK15}, and \eqref{finite-1-moment} extends those in  \cite[Lemma 4]{CK15} and \cite[Theorem 5.12]{HLXZ}.  
\begin{proposition}\label{Le.finite}
For $f=e^{-\phi}\in \LC$ with $o\in \mathrm{int}(\dom(\phi))$ and $\omega\in \mathscr{G}$, one has \begin{align}\label{In.finite} -\infty<
\int_{\mathbb{R}^n}\phi^\ast(\nabla\phi(x))f(x)\omega(x)dx<\infty, \\ 0\leq \int_{\mathbb{R}^n} |\nabla\phi(x)|f(x)\omega(x)dx<\infty \label{finite-1-moment}.
\end{align}
\end{proposition}

 \begin{proof}  The lower bound in \eqref{In.finite}  is an easy consequence of Lemma \ref{Le.finiteness} and   $\phi^*(y)\geq -\phi(o)>-|\phi(o)|>-\infty$  for $y\in\R^n$.
 
Now let us prove the upper bound in \eqref{In.finite}. 

Since $o\in \mathrm{int}(\dom(\phi))$, there exists $\varepsilon_0>0$ such that $\ball( \varepsilon_0)\subseteq \mathrm{int}(\dom(\phi))$. Since $\phi$ is convex, it is  Lipschitz on $\ball(\varepsilon_0)$. Thus, one can find $L>0$ such that $|\nabla\phi|\le L$  almost everywhere on $\ball(\varepsilon_0)$. Moreover, by Lipschitz continuity of $\phi$, we conclude the existence of $c_0\in \mathbb{R}$ such that $\phi\geq c_0$ on $\ball(\varepsilon_0)$. It follows from \eqref{basicequ}, Cauchy-Schwarz inequality, and Lemma \ref{Le.finiteness} that 
\begin{equation}
\label{bounded-above--1a}
	\begin{aligned}  \int_{\ball(\varepsilon_0)}\phi^{\ast}(\nabla\phi(x))e^{-\phi(x)} \omega(x)dx &=\int_{\ball(\varepsilon_0)}\left(\langle x,\nabla\phi\rangle-\phi(x)\right)e^{-\phi(x)}\omega(x) dx\\
&\le(\varepsilon_0 L-c_0)\int_{\ball(\varepsilon_0)} e^{-\phi(x)}\omega(x) dx\\ 
&\le(\varepsilon_0 L+|c_0|)\int_{\ball(\varepsilon_0)} e^{-\phi(x)}\omega(x) dx <\infty. 
\end{aligned} 
\end{equation} 

On the other hand, by \eqref{o-v-estimamte-1}, the continuity of $\omega$, and the lower semi-continuity of $\phi$, one can find a sufficiently large constant $d_1>0$  such that   $\ln \omega(x)\leq \frac{\phi(x)}{2}+d_1$ for all $x\in (\ball(\varepsilon_0))^c$. Due to \eqref{equ-dual-1} and $\phi(o)<\infty$, one has $\phi^*(\nabla\phi(x))\geq -\phi(o)$  for almost all $x\in (\ball( \varepsilon_0))^c$, and hence  
\begin{equation}
\label{upp-b-11}
 \begin{aligned} \int_{(\ball(\varepsilon_0))^c}\!\! \Big(\phi^{\ast}(\nabla\phi(x)) +\phi(o)\Big)e^{-\phi(x)}\omega(x)\,dx 
& \le  \int_{(\ball(\varepsilon_0))^c}\!\!\Big(\phi^{\ast}(\nabla\phi(x)) +\phi(o)\Big)e^{\frac{-\phi(x)}{2}+d_1} dx\\ 
& \le e^{d_1}  \int_{(\ball(\varepsilon_0))^c}\Big(\phi^{\ast}(\nabla\phi(x)) +\phi(o)\Big)e^{-\frac{\phi(x)}{2}} dx\\ 
& \le e^{d_1}  \int_{\R^n}\Big(\phi^{\ast}(\nabla\phi(x)) +\phi(o)\Big)e^{-\frac{\phi(x)}{2}} \,dx\\
& :=e^{d_1}(I_1+\phi(o)I_2), 
\end{aligned}
\end{equation}
 where $I_1$ and $I_2$ are \begin{align} I_1=  \int_{\R^n} \phi^{\ast}(\nabla\phi(x))  e^{-\frac{\phi(x)}{2}} \,dx \ \ \mathrm{and}\ \   I_2= \int_{\R^n}  e^{-\frac{\phi(x)}{2}} \,dx. \end{align}  Clearly, $I_2$ is finite  as implied by \eqref{In.CF13}.   Proposition 7 in \cite{CK15} implies that if $e^{-\phi}\in  \LC$, then  \begin{align}\label{ck-finiteness} -\infty< \int_{\R^n} \phi^{\ast}(\nabla\phi(x)) e^{-\phi(x)} \,dx<\infty.\end{align} Applying this to the function $e^{-\frac{\phi}{2}}\in \LC$ (using the fact that  $I_2$ is finite), and by   \eqref{Eq.(avarphi)*}, one gets \begin{align}
I_1 =\int_{\R^n} \phi^{\ast}(\nabla\phi(x)) e^{-\frac{\phi(x)}{2}} \,dx   = 2\int_{\R^n} \Big(\frac{\phi}{2}\Big)^{\ast}\Big(\frac{\nabla\phi(x)}{2}\Big) e^{-\frac{\phi(x)}{2}} \,dx <\infty. \nonumber \end{align} 
This, together with \eqref{upp-b-11} and Lemma \ref{Le.finiteness},  gives 
\begin{equation}\label{upper-b-2-2-2}
   \begin{aligned} \int_{(\ball(\varepsilon_0))^c}\!\!  \phi^{\ast}(\nabla\phi(x))  e^{-\phi(x)}\omega(x)\,dx 
& \le  e^{d_1} (I_1+\phi(o)I_2)-  \phi(o)\int_{(\ball( \varepsilon_0))^c}  e^{-\phi(x)}\omega(x)\,dx
\\ &\leq 
e^{d_1} (I_1+|\phi(o)|I_2)+|  \phi(o) |  \int_{(\ball(\varepsilon_0))^c} e^{-\phi(x)}\omega(x)\,dx
\\ &\leq  e^{d_1} (I_1+|\phi(o)|I_2)+|  \phi(o) | \cdot \VO(e^{-\phi})<\infty. 
\end{aligned}
\end{equation}
The desired upper bound for \eqref{In.finite} can be obtained by combining \eqref{bounded-above--1a} and \eqref{upper-b-2-2-2}.  

Now let us prove inequality \eqref{finite-1-moment}. The lower bound of \eqref{finite-1-moment} is trivial, and the proof for the upper bound is similar to the one for \eqref{In.finite}. Indeed, recall that one has $|\nabla\phi|\le L$  almost everywhere on $\ball( \varepsilon_0)$. This, together with Lemma \ref{Le.finiteness},  implies  
\begin{align}  \int_{\ball(\varepsilon_0)} |\nabla\phi(x)| e^{-\phi(x)}\omega(x) dx 
&\le L \int_{\ball(\varepsilon_0)} e^{-\phi(x)}\omega(x) dx <\infty. \label{bounded-above--1-moment}  \end{align}  
Recall that   $\ln \omega(x)\leq \frac{\phi(x)}{2}+d_1$ for all $x\in (\ball( \varepsilon_0))^c$. It has been proved in \cite[Lemma 4]{CK15} that, if $e^{-\phi}\in  \LC$, then  \begin{align}\label{ck-finiteness-1-moment}  \int_{\R^n} |\nabla e^{-\phi(x)}| \,dx= \int_{\R^n} |\nabla\phi(x)| e^{-\phi(x)} \,dx<\infty.\end{align} Applying this to the function $e^{-\frac{\phi}{2}}\in \LC$, one gets 
\begin{align}  \int_{(\ball(\varepsilon_0))^c} |\nabla\phi(x)|e^{-\phi(x)}\omega(x)dx  \leq 2e^{d_1}  \int_{(\ball(\varepsilon_0))^c} \frac{|\nabla\phi(x)|}{2} e^{-\frac{\phi(x)}{2}} \,dx   <\infty. \nonumber \end{align} 
This, together with \eqref{bounded-above--1-moment}, gives the upper bound for \eqref{finite-1-moment}.   
 \end{proof} 

\section{The dual Orlicz curvature measures for log-concave functions} \label{Section-4}
In this section, we will establish a variational  formula for the $\omega$-Orlicz moment of the log-concave function $f=e^{-\phi}\in \LO$, which will lead to two Borel measures related to $f$. 
 
\begin{definition}\label{Def.delta,f,g}
Let $\omega\in \mathscr{G}$ and $f, g$ be upper semi-continuous log-concave functions on $\R^n$. Define $\DO(f,g)$, the first variation of the $\omega$-Orlicz moment of $f$ along $g$, with respect to the Asplund sum, by 
\begin{align}
    \DO(f,g)=\lim_{t\to0^{+}}\frac{\VO(f\oplus t\cdot g)-\VO(f)}{t}, \label{def-VO}
\end{align}
if the limit exists.
\end{definition} 
When $\omega$ is the constant function $1$ on $\R^n$, $\DO(f,g)$ recovers $\delta J(f, g)$ in \cite{CF13}. When $\omega(x)=|x|^{q-n}$ for $q>0$, $\DO(f,g)$ reduces to $\delta_q(f, g)$ in \cite{HLXZ}. 

Definition \ref{Def.delta,f,g} is an extension of the geometric case. If $f=\mathbf{1}_K$ and $g=\mathbf{1}_L$ for $K, L\in \mathcal{K}_o^n$, then according to \eqref{Eq.V1,phi(K,L)},
\begin{align*}
\DO(\mathbf{1}_K, \mathbf{1}_L)&= \lim_{t\to0^{+}}\frac{\VO(\mathbf{1}_{K+tL})-\VO(\mathbf{1}_K)}{t}\\ &=\lim_{t\to0^{+}}\frac{\VO(K+tL)-\VO(K)}{t} =\widetilde{V}_{1,\omega}(K, L).
\end{align*} 
However, we shall see that for generic log-concave functions $f$ and $g$, establishing the existence of the limit in \eqref{def-VO} and computing the limit turn out to be quite different from their geometric counterparts. See, for example, \cite{CF13,   FXY20+, HLXZ, Rot20, Rot22}. 

Recall that $K_f=\overline{\{x\in \R^n: f(x)\neq 0\}}$ denotes the closure of the support of $f$. We first state our main theorem, the proof of which will be carried out throughout this section.

\begin{theorem}\label{Th.variation} Let $\omega\in \mathscr{G}$. Assume $f=e^{-\phi}\in \LO$ such that 
\begin{align}\label{condition-sup<infty}
\limsup_{|x|\to 0} \frac{|f(x)-f(o)|}{|x|^{\alpha+1}}<\infty,
\end{align} 
for some $0<\alpha<1.$ The following arguments hold. 
\begin{enumerate}[i)]
	\item \label{statement 1} If $g=e^{-\psi}=c \mathbf{1}_L$ for some $c>0$ and some compact convex set $L\subset \R^n$ with $o\in L$ (or, equivalently,  $\psi=-\ln c+\chi_L$), then  
\begin{align}\label{Main.Variation-indicator}
\DO(f, g)=\int_{\mathbb{R}^n} \psi^\ast(\nabla \phi(x)) e^{-\phi(x)}\omega(x)\,dx +
\int_{\partial K_f}h_{L}(\nu_{K_f}(x) )e^{-\phi(x)} \omega(x)\,d\mathcal{H}^{n-1}(x).
\end{align}  
	\item \label{statement 2} If  $g=e^{-\psi}$ is an upper semi-continuous, log-concave function with compact support and $g(o)>0$ (in this case, $K_g$ is a compact set such that $o\in K_g$), then
\begin{align}\label{Main.Variation-new}
\DO(f,g)=\int_{\mathbb{R}^n}\psi^\ast(\nabla \phi(x)) e^{-\phi(x)}\omega(x)\,dx +
\int_{\partial K_f}h_{K_g}(\nu_{K_f}(x) ) e^{-\phi(x)}\omega(x)\,d\mathcal{H}^{n-1}(x).
\end{align}
\end{enumerate}
\end{theorem}

We need the following trivial lemma.
\begin{lemma}\label{Le.LC}
Let $f\in \LC$ and $g$ be a nonzero upper semi-continuous log-concave function with finite $L^1$ norm. Then for any $s,t>0$, we have $s\cdot f\oplus t\cdot g \in \LC$.
\end{lemma}
\begin{proof}
It follows directly from the definition of Asplund sum that $s\cdot f\oplus t\cdot g$ is an upper semi-continuous, log-concave function. By \cite[Proposition 2.6]{CF13}, it follows that $s\cdot f\oplus t\cdot g$ has finite $L^1$ norm. Hence, it only remains to show that $s\cdot f\oplus t\cdot g$ has a non-zero $L^1$ norm. This is trivial to see, since, if $g(x_0)>0$, then
\begin{equation}
	(s\cdot f\oplus t\cdot g)(x)=\sup_{x=x_1+x_2}f\left(\frac{x_1}{s}\right)^sg\left(\frac{x_2}{t}\right)^t\ge f\left(\frac{x-tx_0}{s}\right)^sg\left(x_0\right)^t.
\end{equation}
This implies that $(s\cdot f\oplus t\cdot g)$ is not almost everywhere zero and consequently, $s\cdot f\oplus t\cdot g \in \LC$.
\end{proof}
 
We will need the following point-wise variational formula established in \cite[Lemma 3.1]{Rot22}: If $f=e^{-\phi}\in \LC$ and $g=e^{-\psi}$ is an upper semi-continuous, log-concave function with compact support, then for almost all $x\in \R^n$, we have
\begin{align}\label{Eq.CF13}
\frac{d}{dt} \big(f\oplus t\cdot g \big)(x)\bigg|_{t=0^+} =\psi^\ast(\nabla \phi(x))e^{-\phi(x)}.
\end{align}

The following lemma shows that to establish Theorem \ref{Th.variation}, it suffices to show Statement \ref{statement 1}) in Theorem \ref{Th.variation}. 
\begin{lemma}\label{i-to-ii}
If Statement \ref{statement 1}) holds in Theorem \ref{Th.variation}, then Statement \ref{statement 2}) holds in Theorem \ref{Th.variation}.
\end{lemma}
\begin{proof}
Assume $g=e^{-\psi}$ is an upper semi-continuous, log-concave function with compact support and $g(o)>0$. 
Since $g$ is compactly supported, there exists $C>0$ such that $g\leq C\mathbf{1}_{K_g}$. Note that $K_g$ is a compact convex set with $o\in K_g$. 
For each $i\in \mathbb{N}$, let 
$$K_i=\{x\in\mathbb{R}^n: g(x)\ge 1/i\} \subset K_g.$$
 Since $g(o)>0$, there exists $i_0>0$ such that for every $i\geq i_0$, the set $K_i$ is a compact convex set containing the origin. Moreover, $\overline{\cup_{i=i_0}^\infty K_i}=K_g$.
 
Let  $\widetilde{\psi}=-\ln C+\chi_{K_g}$ and $\psi_i=\ln i+\chi_{K_i}$ for each $i\geq i_0$ and set
$$f_t=f\oplus t\cdot g, \ \  \widetilde{f}_t=f\oplus t\cdot e^{-\widetilde{\psi}} \ \ \mathrm{and} \ \ f_{i, t}= f\oplus t\cdot e^{-\psi_i}.$$
Note that $f_{i,t}\leq f_t\leq \widetilde{f}_t$. From Lemma \ref{Le.LC}, one can see that $f_t$, $\widetilde{f}_t$, and $f_{i, t}$ are all belong to $\LC$ for any $t>0$. By Lemma \ref{Le.finiteness}, $\VO(\widetilde{f}_t)$ and $\VO(f_{i,t})$ are all positive and finite.
It follows from Fatou's lemma that
\begin{align*} \liminf_{t\to 0^+}\frac{\VO(\widetilde{f}_t)-\VO(f_t)}{t} =\liminf_{t\to 0^+}\int_{\mathbb{R}^n}\frac{\widetilde{f}_t-f_t}{t}\omega dx  \ge\int_{\R^n} \liminf_{t\to 0^+}\frac{\widetilde{f}_t-f_t}{t}\omega dx.
\end{align*} 
Applying  \eqref{Eq.CF13} to $f_t$ and $\widetilde{f}_t$, respectively, one gets   
\begin{align}
   \liminf_{t\to 0^+}\frac{\VO(\widetilde{f}_t)-\VO(f_t)}{t}\ge\int_{\R^n} \liminf_{t\to 0^+}\frac{\widetilde{f}_t-f_t}{t}\omega dx=\int_{\mathbb{R}^n} \Big(\widetilde{\psi}^\ast(\nabla \phi) -\psi^\ast(\nabla\phi)\Big)  e^{-\phi}\omega dx.\label{Fatou-1-1}
\end{align} 
A similar calculation shows that, for each $i\geq i_0$, 
\begin{align}
    \liminf_{t\to 0^+}\frac{\VO( f_t)-\VO(f_{i, t})}{t}
 \ge \int_{\mathbb{R}^n}\Big(\psi^\ast(\nabla\phi)-\psi_i^\ast(\nabla \phi)\Big)    e^{-\phi}\omega\,dx. \label{liminf--111-22}
\end{align} 
Applying \eqref{Main.Variation-indicator}  to $g=e^{-\widetilde{\psi}}$, one can get  
\begin{align}\DO(f, \widetilde{g}) 
&=\lim_{t\to 0^+} \frac{\VO(\widetilde{f}_t)-\VO(f)}{t}\nonumber \\
&=\int_{\mathbb{R}^n} \widetilde{\psi}^*(\nabla \phi(x)) e^{-\phi(x)}\omega(x)\,dx +
\int_{\partial K_f}h_{K_g}(\nu_{K_f}(x)) e^{-\phi(x)}\omega(x)\,d\mathcal{H}^{n-1}(x).\label{variation1-deltag}
\end{align} 
 This, together with  \eqref{Fatou-1-1}, implies that
\begin{equation}\label{lim-sup-1}
\begin{aligned} \DO(f,g)&=\lim_{t\to0^{+}}\frac{\VO(f\oplus t\cdot g)-\VO(f)}{t}\\   
&\leq  \limsup_{t\to 0^+}\frac{\VO(\widetilde{f}_t)-\VO(f)}{t}-\liminf_{t\to 0^+}\frac{\VO(\widetilde{f}_t)-\VO(f_t)}{t}\\ 
&\leq \int_{\mathbb{R}^n}  \psi^*(\nabla \phi(x)) e^{-\phi(x)}\omega(x)\,dx +
\int_{\partial K_f}h_{K_g}(\nu_{K_f}(x)) e^{-\phi(x)}\omega(x)\,d\mathcal{H}^{n-1}(x). 
\end{aligned}
\end{equation}
Applying \eqref{Main.Variation-indicator} to $g=e^{-\psi_i}$ for each $i\geq i_0$, and with the help of \eqref{liminf--111-22}, one can get
\begin{align*} \DO(f,g)  &\geq \liminf_{t\to 0^+}\frac{\VO(f_{i, t})-\VO(f)}{t}+\liminf_{t\to 0^+}\frac{\VO( f_t)-\VO(f_{i, t})}{t}\nonumber \\ & =\int_{\mathbb{R}^n}  \psi^*(\nabla \phi(x)) e^{-\phi(x)}\omega(x)\,dx +
\int_{\partial K_f}h_{K_i}(\nu_{K_f}(x)) e^{-\phi(x)}\omega(x)\,d\mathcal{H}^{n-1}(x).
\end{align*}  Note that $\overline{\bigcup_{i=i_0}^\infty K_i}=K_g$, and thus $h_{K_i}$ increases to its supremum $\sup_{i\geq i_0} h_{K_i}=h_{K_g}$.

It follows from the monotone convergence theorem  that
\begin{align} \DO(f,g) \geq  \int_{\mathbb{R}^n}  \psi^*(\nabla \phi(x)) \omega(x)e^{-\phi(x)}\,dx +
\int_{\partial K_f}h_{K_g}(\nu_{K_f}(x)) \omega(x)e^{-\phi(x)}\,d\mathcal{H}^{n-1}(x). \label{lim-inf-1}
\end{align} 
Statement \ref{statement 2}) now follows from \eqref{lim-sup-1} and \eqref{lim-inf-1} (if Statement \ref{statement 1})  is assumed).
\end{proof}

The next lemma reduces Statement \ref{statement 1}) in Theorem \ref{Th.variation} to the special case $c=1$. 
\begin{lemma}\label{c=1}
If Statement \ref{statement 1}) holds in Theorem \ref{Th.variation} for $c=1$, then Statement \ref{statement 1}) in Theorem \ref{Th.variation} holds for any $c>0$.
\end{lemma}
\begin{proof}
Let $L$ be a compact convex subset containing
the origin. Assume that  \eqref{Main.Variation-indicator} holds for $\mathbf{1}_{L}=e^{-\chi_{L}}$, that is,
\begin{align} 
\DO(f, \mathbf{1}_{L})=
     \int_{\mathbb{R}^n} \chi_{L}^\ast(\nabla \phi(x)) e^{-\phi(x)}\omega(x)\,dx +
\int_{\partial K_f}h_{L}(\nu_{K_f}(x) ) e^{-\phi(x)}\omega(x)\,d\mathcal{H}^{n-1}(x). \label{special-c-L}
\end{align}
Let $c>0$ and ${g}=e^{-{\psi}_0}=c\mathbf{1}_{L}$.
Note that $({\psi}_0)^\ast=\chi_{L}^\ast+\ln c$ and $K_{{g}}=L$. By definition, $f\oplus t\cdot {g}=c^t (f \oplus t\cdot \mathbf{1}_{L})$ and
\begin{align*}
\DO(f,c \mathbf{1}_{L}) &=\lim_{t\to 0^+}\frac{\VO(f\oplus t\cdot {g})-\VO(f)}{t}\\&=\lim_{t\to 0^+}\frac{c^t \VO(f\oplus t\cdot \mathbf{1}_{L})-\VO(f)}{t}\\
&= \VO(f)\cdot \ln c+\DO(f, \mathbf{1}_{L}). 
\end{align*}
By \eqref{relation-supp-dual} and \eqref{special-c-L}, one further has
\begin{align*}
\DO(f,c \mathbf{1}_{L}) &= \VO(f)\cdot \ln c+\DO(f, \mathbf{1}_{L})\\
&=\VO(f)\cdot \ln c+\int_{\mathbb{R}^n} h_{L}(\nabla \phi(x)) e^{-\phi(x)}\omega(x)\,dx +
\int_{\partial K_f}h_{L}(\nu_{K_f}(x) ) e^{-\phi(x)}\omega(x)\,d\mathcal{H}^{n-1}(x)\\
&=\int_{\mathbb{R}^n} ({\psi}_0)^\ast(\nabla \phi(x)) e^{-\phi(x)}\omega(x)\,dx +
\int_{\partial K_f}h_{L}(\nu_{K_f}(x) ) e^{-\phi(x)}\omega(x)\,d\mathcal{H}^{n-1}(x). 
\end{align*}
This completes the proof.
\end{proof}

For $s>0$, recall that $E_s(f)=\{x\in \R^n: f(x)\geq s\}$ is the superlevel set of $f$. Let $f=e^{-\phi}\in \LC$ and $L$ be a compact convex set. By the definition of the Asplund sum, 
\begin{equation}
\label{eq 81614}
\begin{aligned}
	E_s(f\oplus t\cdot \mathbf{1}_L) &= \{x\in \R^n: e^{-\inf_{y\in \R^n}(\phi(x-y)+\chi_{tL}(y))}\geq s\}\\
	&= \{x\in \R^n: e^{-\inf_{y\in tL}\phi(x-y)}\geq s\}\\
	&= \{x\in \R^n: \sup_{y\in tL}e^{-\phi(x-y)}\geq s\}\\
	&= \bigcup_{y\in tL} \{x+y\in \R^n: e^{-\phi(x)}\geq s\}\\
	&= E_s(f)+tL.
\end{aligned}	
\end{equation}
Note that if $g$ is a generic log-concave function, there is an (albeit much more complicated) formula for the superlevel sets for $f\oplus t\cdot g$. 

For simplicity, for $f\in \LO$, we write $M_f=f(o)$, which is the maximum of $f$. For the rest of the section, when we write $E_s$, we exclusively mean $E_s(f)$.

The following two lemmas are key  results to establish the formula (\ref{Main.Variation-indicator}).
\begin{lemma}\label{Le.DCT1}
Let $\omega \in \mathscr{G}$, $f\in \LO$ and $L$ be a compact convex set containing the origin. If there exist $0<\delta_0<M_f$ and $c_0,c_1>0$ such that  
\begin{equation}
\label{eq 8161}
	c_0(M_f-s)^{\frac{1}{1+\alpha}} B_2^n\subset E_s\subset c_1B_2^n
\end{equation}
for some $\alpha\in (0,1)$ and all $s\in (M_f-\delta_0, M_f)$, then
\begin{align}\label{dominated-1}
\lim_{t\rightarrow 0^+} \int_{M_f-\delta_0}^{M_f} \frac{\widetilde{V}_\omega(E_s+tL)-{\widetilde{V}_\omega (E_s)}}{t}ds =\int_{M_f-\delta_0}^{M_f}   \lim_{t\rightarrow 0^+} \frac{\widetilde{V}_\omega (E_s+tL)-\widetilde{V}_\omega (E_s)}{t}ds <\infty.
\end{align}
\end{lemma}
\begin{proof}
We need to find a dominating function to apply the dominated convergence theorem.

Let $s\in (M_f-\delta_0, M_f)$ be arbitrarily fixed and $t\in (0,1)$. Note that by Lemma \ref{Le.var.geo}  and  the Lagrange mean value theorem there exists $\theta\in (0,t)\subset (0,1)$ such that 
\begin{equation}\label{Eq.Lagrange}
\begin{aligned}
\frac{\VO(E_s+tL)-\VO(E_s)}{t}
&=\frac{\partial \VO(E_s+tL)}{\partial t}\bigg|_{t=\theta} 
\\&=\int_{S^{n-1}}\frac{h_L(u)}{h_{E_s+\theta L}(u)}d\widetilde{C}_\omega(E_s+\theta L,u).
\end{aligned} 
\end{equation}
Note that by \eqref{eq 8161} and the fact that $L$ contains the origin, on $S^{n-1}$, we have
\begin{equation}
\label{eq 8162}
	h_{E_s+\theta L}\geq h_{E_s}\geq c_0(M_f-s)^{\frac{1}{1+\alpha}}.
\end{equation}
On the other hand, we have
\begin{equation}
\label{eq 8163}
	E_s+\theta L\subset c_1\ball+L\subset c_2\ball,
\end{equation}
for some $c_2>0$ depending only on $L$. Here, we used the fact that $L$ is compact.

Combining \eqref{Eq.Lagrange}, \eqref{eq 8162}, and \eqref{eq 8163}, we have
\begin{equation}
\begin{aligned}
	\frac{\VO(E_s+tL)-\VO(E_s)}{t}&\leq \frac{\max_{S^{n-1}}h_L}{c_0(M_f-s)^{\frac{1}{1+\alpha}}} \widetilde{C}_\omega(E_s+\theta L, \sphere)\\
	&=\frac{\max_{S^{n-1}} h_L}{c_0(M_f-s)^{\frac{1}{1+\alpha}}} \int_{S^{n-1}}\big(\rho_{E_s+\theta L}(u)\big)^n\omega (\rho_{E_s+\theta L}(u)u)du.
\end{aligned}
\end{equation}
Condition (A2) now implies the existence of $c_3>0$ such that 
\begin{equation}
\label{eq 8165}
	\frac{\VO(E_s+tL)-\VO(E_s)}{t}\leq c_3 (M_f-s)^{-\frac{1}{1+\alpha}}.
\end{equation}
Since $\alpha\in (0,1)$, the right-hand side of \eqref{eq 8165} is integrable. The desired result now follows from the dominated convergence theorem.
\end{proof}

\begin{lemma}\label{Le.DCT2}
Let $\omega \in \mathscr{G}$, $f\in \LO$ and $L$ be a compact convex set containing the origin. If there exist $0<\delta_0<M_f$ and $c_0>0$ such that 
\begin{equation}
\label{eq 8167}
	c_0B_2^n\subset E_s,
\end{equation}
 for all $s\in (0, M_f- \delta_0)$, then,
\begin{equation}\label{dominated-2}
\lim_{t\rightarrow 0^+} \int_0^{M_f-\delta_0} \frac{\widetilde{V}_\omega (E_s+tL)-{\widetilde{V}_\omega (E_s)}}{t}ds =\int_0^{M_f-\delta_0}   \lim_{t\rightarrow 0^+} \frac{\widetilde{V}_\omega (E_s+tL)-{\widetilde{V}_\omega (E_s)}}{t}ds <\infty.
\end{equation}
\end{lemma}
\begin{proof}
As in the proof of Lemma \ref{Le.DCT1}, we need to find a dominating function. 
Let $s\in(0,M_f-\delta_0)$. By \eqref{eq 8167}, it can be shown (see, \emph{e.g.},  \cite[Corollary 4.4]{HLXZ}) that there exist $t_0>0$ and $c_1>0$ depending only on $c_0$ and $\max_{x\in L}|x|$ such that for every $t\in (0,t_0)$ and $u\in \sphere$, we have 
\begin{align}\label{Le.HLXZ}
\bigg|\frac{\rho_{E_s+tL}(u)-\rho_{E_s}(u)}{t}\bigg|< c_1\rho_{E_s}(u).
\end{align}
In particular, by \eqref{eq 8167} and the fact that $L$ is a compact set containing the origin, we may assume $t_0$ is sufficiently small so that for every $t\in (0,t_0)$, we have $c_0\ball\subset E_s+tL\subset 2E_s$. 

For each $t\in (0, t_0)$, the definition of $\overline{\omega}(t,u)$ (see \eqref{overline-omega}), the fundamental theorem of calculus, and \eqref{Le.HLXZ} imply the existence of $\vartheta=\vartheta(s,t, u)\in(\rho_{E_s}(u), \rho_{E_s+tL}(u))\subset (\rho_{E_s}(u), 2\rho_{E_s}(u))$ such that
\begin{equation}
	\label{eq 8168}
	\begin{aligned}
		&\bigg|\frac{\overline{\omega}(\rho_{E_s+tL}(u),u)-\overline{\omega}(\rho_{E_s}(u),u)}{t}\bigg|\\
=&\bigg|\frac{\overline{\omega}(\rho_{E_s+tL}(u),u)-\overline{\omega}(\rho_{E_s}(u),u)}{\rho_{E_s+tL}(u)-\rho_{E_s}(u)}\bigg|\times
\bigg|\frac{\rho_{E_s+tL}(u)-\rho_{E_s}(u)}{t}\bigg|\\
\le& c_1\rho_{E_s}(u)\cdot    \omega(\vartheta u)\vartheta ^{n-1}
	\end{aligned}
\end{equation} 
for almost all $u\in \sphere$, $t\in (0,t_0)$, and $s\in (0, M_f-\delta_0)$. In particular, by $\eqref{eq 8167}$, we have $\vartheta> c_0$. 

 Since $f\in \LC$, by \eqref{In.CF13}, there exist $c_2>0$ and $c_3\in \R$ such that
\begin{equation}
	E_s = \{x\in \R^n: e^{-\phi}\geq s\}\subset \{x\in \R^n: e^{-c_2|x|-c_3}\geq s\},
\end{equation}
which implies that 
\begin{equation}
\label{eq 81610}
	\rho_{E_s}\leq \frac{-c_3-\ln s}{c_2}.
\end{equation}

Condition (A3), together with the continuity of $\omega$ on $\R^n\setminus \{o\}$, implies that there exists $c_4>0$ such that, for every $x\in \R^n$ with $|x|\geq c_0$,
\begin{equation}
\label{eq 8169}
	\ln \omega(x)\leq \frac{c_2}{4} |x|+c_4.
\end{equation} 
By polar coordinates,
 \begin{align}\label{GDVofK-Es}
\VO(E_s+tL) =\int_{S^{n-1}}\overline{\omega}(\rho_{E_s+tL}(u),u)du.\end{align}
Hence, for every $t\in (0,t_0)$ and $s\in (0, M_f-\delta_0)$, by \eqref{eq 8168}, \eqref{eq 8169}, and the fact that $\vartheta\in (\rho_{E_s}(u), 2\rho_{E_s}(u))$,
\begin{equation}
\label{DCT2.1}
	\begin{aligned}
		\frac{\widetilde{V}_\omega(E_s+tL)-\widetilde{V}_\omega(E_s)}{t}&=
\int_{S^{n-1}}\frac{\overline{\omega}(\rho_{E_s+tL}(u),u)-\overline{\omega}(\rho_{E_s}(u),u)}{t}du\\
&\le c_1 \int_{S^{n-1}}\rho_{E_s}(u) \omega(\vartheta u)\vartheta^{n-1}du\\
&\le c_1\int_{S^{n-1}}\rho_{E_s}(u)\vartheta^{n-1}e^{\frac{c_2}{4}\vartheta+c_4}du\\
&\le 2^{n-1}c_1e^{c_4}\int_{S^{n-1}}\rho_{E_s}(u)^{n}e^{\frac{c_2}{2}\rho_{E_s}(u)}du.
	\end{aligned}
\end{equation}

We now check that the function 
\begin{equation}
	s\mapsto \int_{S^{n-1}}\rho_{E_s}(u)^{n}e^{\frac{c_2}{2}\rho_{E_s}(u)}du
\end{equation}
is integrable on $(0, M_f-\delta_0)$. Indeed, by \eqref{eq 81610}, we have
\begin{align*}
	\int_{S^{n-1}}\rho_{E_s}(u)^{n}e^{\frac{c_2}{2}\rho_{E_s}(u)}du&\leq c_2^{-n}\int_{S^{n-1}} (-c_3-\ln s)^n e^{-\frac{c_3+\ln s}{2}}du\\
 &= c_2^{-n}e^{-\frac{c_3}{2}}nV_n(\ball) (-c_3-\ln s)^n e^{-\frac{\ln s}{2}}. 
\end{align*}
It is simple to check the right-hand side of the above equation is integrable on $(0, M_f-\delta_0)$. 
\end{proof}

\begin{lemma}
\label{Le.DCT main}
Let $\omega\in\mathscr{G}$ and $f\in\LO$. If $f$ satisfies \eqref{condition-sup<infty} for some $\alpha\in (0,1)$, then for every compact convex $L$ containing the origin, we have
\begin{equation}
	 \DO(f, \mathbf{1}_{L})=\int_0^{M_f}   \widetilde{V}_{1,\omega}(E_s,L)ds <\infty.
\end{equation}
\end{lemma}
\begin{proof}
	By layer-cake representation and \eqref{eq 81614}, we have 
	\begin{equation}
	\label{eq 81618}
	\begin{aligned}
		\DO(f, \mathbf{1}_{L}) &= \lim_{t\rightarrow 0^+} \int_{0}^{M_f} \frac{\widetilde{V}_{\omega}(E_s(f\oplus t\cdot \mathbf{1}_L))-\widetilde{V}_{\omega}(E_s(f))}{t}ds\\
		&= \lim_{t\rightarrow 0^+} \int_{0}^{M_f} \frac{\widetilde{V}_{\omega}(E_s(f)+tL)-\widetilde{V}_{\omega}(E_s(f))}{t}ds.
	\end{aligned}	
	\end{equation}
	
	Condition \eqref{condition-sup<infty} implies the existence of $\delta>0$ and $c_1>0$ such that 
	\begin{equation}
		f(x)\geq M_f-c_1|x|^{\alpha+1}\ \text{ on }\ \ball(\delta).
	\end{equation}
	This implies
	\begin{equation}
		E_s\supset \min\left\{\left(\frac{M_f-s}{c_1}\right)^{\frac{1}{\alpha+1}}, \delta\right\}\ball.
	\end{equation}
	We can choose $\delta_0>0$ sufficiently small so that for every $s\in (M_f-\delta_0, M_f)$, one has 
	\begin{equation}
		E_s\supset \left(\frac{M_f-s}{c_1}\right)^{\frac{1}{\alpha+1}}\ball.
	\end{equation}
	Moreover, since $f\in \LO$, there exists $c_2>0$ such that $E_s\subset c_2 \ball$ for every $s\in (M_f-\delta_0, M_f)$. On the other hand, by definition of superlevel sets, for every $s\in (0, M_f-\delta_0)$, there exists $c_3>0$ such that
	\begin{equation}
		E_s\supset E_{M_f-\delta_0}\supset c_3\ball.
	\end{equation}
	
	By \eqref{eq 81618}, Lemma \ref{Le.DCT1}, and Lemma \ref{Le.DCT2}, we have
	\begin{equation}
		\DO(f, \mathbf{1}_{L})= \int_{0}^{M_f} \lim_{t\rightarrow 0^+}\frac{\widetilde{V}_{\omega}(E_s(f)+tL)-\widetilde{V}_{\omega}(E_s(f))}{t}ds<\infty.
	\end{equation}
	The desired result now follows from Lemma \ref{Le.var.geo}.
\end{proof}

The following corollary follows immediately from the definition of $\widetilde{V}_{1,\omega}$ and Lemma \ref{Le.DCT main}.
\begin{corollary}\label{coro 8161}
	Let $\omega\in\mathscr{G}$ and $f\in\LO$. If $f$ satisfies \eqref{condition-sup<infty} for some $\alpha\in (0,1)$, then for every compact convex $L_1, L_2$ containing the origin, we have
	\begin{align}\label{linear}
\DO(f, \mathbf{1}_{L_1+L_2})=
\DO(f, \mathbf{1}_{L_1})+\DO(f, \mathbf{1}_{L_2}).
\end{align}
\end{corollary}

We now show that the proof of Statement \ref{statement 1}) in Theorem \ref{Th.variation} can be reduced to the special case where $c=1$ and $L\in  \mathcal{K}_o^n$. 

\begin{lemma}\label{cor-statement-i}
If Statement \ref{statement 1}) in Theorem \ref{Th.variation} holds for $c=1$ and $L\in\mathcal{K}_o^n$, then Statement i) in Theorem \ref{Th.variation} holds.
\end{lemma}
\begin{proof}
Assume Statement \ref{statement 1}) in Theorem \ref{Th.variation} holds for $c=1$ and any $Q\in\mathcal{K}_o^n$.

By Lemma \ref{c=1}, it suffices to show the validity of Statement \ref{statement 1}) in Theorem \ref{Th.variation} with $c=1$ and $L$ a compact convex set containing the origin. Set $Q=L+B_2^n\in \mathcal{K}_o^n$. Then, by Corollary \ref{coro 8161},  we have
\begin{equation}
	\begin{aligned}
		\DO({f,\mathbf{1}_{L}})+\DO({f,\mathbf{1}_{B_2^n}})&=\DO({f,\mathbf{1}_{Q}})\\
		&=\int_{\mathbb{R}^n} h_{Q}(\nabla \phi)\omega e^{-\phi}\,dx +
\int_{\partial K_f}h_{Q}(\nu_{K_f}) \omega e^{-\phi}\,d\mathcal{H}^{n-1}(x)\\
		&= \int_{\mathbb{R}^n} h_{L}(\nabla \phi) \omega e^{-\phi}\,dx +
\int_{\partial K_f}h_{L}(\nu_{K_f}) \omega e^{-\phi}\,d\mathcal{H}^{n-1}(x) + \DO({f,\mathbf{1}_{B_2^n}}).
	\end{aligned}
\end{equation}
This concludes the proof.
\end{proof}

The last piece of the proof of Theorem \ref{Th.variation} is the next lemma.
\begin{lemma}\label{lemma 8162}
 Statement i) in Theorem \ref{Th.variation} holds when $c=1$ and $L\in \mathcal{K}_o^n$.
\end{lemma}
\begin{proof}
By \eqref{eq 81620}, \eqref{eq yz1}, and \eqref{eq 9.12.1}, we have 
\begin{align}
\label{eq 9.12.3}
\operatorname{Per}_{L,\omega}(\partial E_s)=TV_{L,\omega}(\mathbf{1}_{E_s})=
\int_{\partial E_s}h_L(\nu_{E_s}(x))\omega(x)d\mathcal{H}^{n-1}(x)
=\widetilde{V}_{1,\omega}(E_s,L).
\end{align}
It now follows from Lemma \ref{Le.DCT main}, \eqref{eq 9.12.3}, \eqref{coarea-formula}, \eqref{eq yz1}, and \eqref{eq 81620} that
\begin{equation}
	\widetilde{\delta}_{\omega}(f,\mathbf{1}_L)=TV_{L,\omega}(f)= \int_{\mathbb{R}^n}h_L(\nabla \phi) f\omega dx+
\int_{\partial K_f}h_L(\nu_{K_f}) f \omega d\mathcal{H}^{n-1}(x).
\end{equation}
\end{proof}

Theorem \ref{Th.variation} now follows from Lemmas \ref{i-to-ii}, \ref{cor-statement-i}, and \ref{lemma 8162}. Theorem \ref{Th.variation} motivates the following definition.
\begin{definition}\label{definition 4.4} Let $f=e^{-\phi}\in \LO$ and $\omega\in\mathscr{G}$. The Euclidean dual $\omega$-Orlicz curvature measure, with respect to the weight function $\omega$,  of $f=e^{-\phi}$, denoted by $\Cure(f, \cdot)$, is the Borel measure on $\R^n$ defined by: for any Borel subset $E\subset \R^n$, \begin{align}
    \label{eud-cur-1} 
    \Cure(f, E)=\int_{\{x\in \R^n:\ \nabla\phi(x)\in E\}}e^{-\phi(x)}\omega(x)\,dx=\int_{\{x\in \dom(\phi):\ \nabla\phi(x)\in E\}}e^{-\phi(x)}\omega(x)\,dx. 
\end{align} 
The spherical dual $\omega$-Orlicz curvature measure,  with respect to the weight function $\omega$,   of $f=e^{-\phi}$, denoted by $\Curs(f, \cdot)$, is the Borel measure on $\sphere$ defined by: for any Borel subset $\eta\subset \sphere$, 
\begin{align}
    \label{sphere-cur-1} 
    \Curs(f, \eta)=\int_{\{x:\partial K_f:\  \nu_{K_f}(x)\in \eta\}} e^{-\phi(x)}\omega(x)\,d\mathcal{H}^{n-1}(x). 
\end{align} 
\end{definition} 

With the above notations, one can rewrite \eqref{Main.Variation-new} as \begin{align}\label{Main.Variation}
\DO(f,g)=\int_{\mathbb{R}^n}\psi^\ast(y)d\Cure(f,y)+
\int_{S^{n-1}}h_{K_g}(v)d\Curs(f,v).
\end{align}

\section{The dual Orlicz Minkowski problem for log-concave functions}\label{Se.dualMin}

Throughout the section, we assume $\omega \in \mathscr{G}$ is fixed. The section focuses on the functional dual Orlicz Minkowski problem. 

\begin{problem}[The dual Orlicz Minkowski problem for $\Cure(f, \cdot)$] \label{complete-2} Given a nonzero finite Borel measure $\mu$ on $\R^n$. What are the necessary and sufficient conditions on $\mu$ so that there exist $f\in \LC$  such that $\mu=\Cure(f, \cdot)?$
\end{problem}
In particular, we focus our attention on the origin-symmetric case: $\mu$ is an even measure on $\R^n$, and $f,\omega$ are even functions.
It is worth pointing out that a similar problem can be posed for the measures $\Cure(f, 
\cdot)$ and $\Curs(f, \cdot)$ simultaneously. Given a pair of nonzero finite Borel measures $\mu$ on $\R^n$ and $\nu$ on $\sphere$, find the necessary and sufficient conditions on $\mu$ and $\nu$ so that $f\in \LC$ with $\mu= \Cure(f, \cdot)$ and  $\nu= \Curs(f, \cdot)$.

Let $\mathscr{M}$ be the set of nonzero even finite Borel measures on $\mathbb{R}^n$ not concentrated in any proper subspace of $\R^n$ with finite first moment.  Let $M_\mu$ be the interior of the convex hull of the closed support set of $\mu$, i.e.,
\begin{equation}
	M_{\mu}= \text{int} \big(\text{conv}(\text{supp}\,\mu)\big).
\end{equation}
It is simple to see that if $\mu\in \mathscr{M}$, then $o\in M_\mu$.  We point out that if $\phi$ is a $\mu$-integrable convex function, then $\phi$ must be finite on $M_\mu$. In particular, this means $\phi$ is finite in a neighborhood of the origin.

The following lemma, shown in \cite[Proposition 5.4]{Rot20}, is crucial. 

\begin{lemma}\label{rotem}
Let $\phi:\mathbb{R}^n\to
\mathbb{R}\cup\{+\infty\}$ be a lower semi-continuous function with  $\phi(o)<\infty$. Assume that $g: \mathbb{R}^n\to\mathbb{R}$ is bounded and continuous. Then at every point $x\in\mathbb{R}^n$ where $\phi^*$ is differentiable we have
\begin{equation}
\frac{d}{dt}\Big|_{t=0}(\phi+tg)^*(x)=-g(\nabla \phi^*(x)).
\end{equation}
\end{lemma}

We also need the following lemma.

\begin{lemma}\label{Le.roc}
\cite[Theorem 10.9]{Roc70}
Let $\Omega\subset \R^n$ be an open convex set, and $\{f_i\}_{i\ge1}$ be a sequence of finite convex functions on $\Omega$. If $\{f_i\}_{i\ge1}$ is bounded for each $x\in \Omega$, then there exists a subsequence $\{f_{i_j}\}_{j\ge1}$  that converges to a finite convex function $f$ pointwisely on $\Omega$ and uniformly on any closed bounded subset of $\Omega$.
\end{lemma}

Denote by $\LP$ the set of all even non-negative integrable functions with respect to $\mu$ on $\mathbb{R}^n$. For any fixed $a>0$, consider the following optimization problem: 
\begin{align}\label{optim.P}
\Theta_a=\inf \left\{\int_{\mathbb{R}^n} \phi d \mu: \phi \in \LP,  \
\widetilde{V}_\omega\left(e^{-\phi^*}\right) \geq a
\right\}.
\end{align}

The following lemma is crucial for finding the optimizers for the optimization problem \eqref{optim.P}. The proof for this lemma essentially follows the same way as that in \cite[Lemma 17]{CK15}.

\begin{lemma}\label{Le.existence}
Let $\omega\in\mathscr{G}$ and $\mu\in\mathscr{M}$. If $\phi_i\in\LP$ is convex and
\begin{align}\label{In-4}
\sup_{i\in\mathbb{N}}\int_{\mathbb{R}^n}\phi_i(x)d\mu(x)<\infty,
\end{align}
then there exist a subsequence $\{\phi_{i_j}\}_{j\in\mathbb{N}}$ of $\{\phi_i\}_{i\in \mathbb{N}}$ and a convex function $\phi_0\in\LP$ such that
\begin{align}\label{two-terms}
\int_{\mathbb{R}^n}\phi_0 d\mu\le\liminf_{j\to\infty}\int_{\mathbb{R}^n}\phi_{i_j}d\mu
\quad\text{and}\quad
\widetilde{V}_\omega(e^{-\phi_0 ^\ast})\ge
\limsup_{j\to\infty}\widetilde{V}_\omega(e^{-\phi_{i_j}^\ast}).
\end{align}
\end{lemma}

\begin{proof} Applying \eqref{In-4} and \cite[Lemma 16]{CK15} to $\mu \in \mathscr{M}$ and $x_0 \in M_{\mu}$, one can find a constant $C_{\mu,x_0} >0$ such that for any $i\in \mathbb{N}$,
\begin{align}\label{CK'In}
\phi_i(x_0)\le C_{\mu, x_0} \int_{\mathbb{R}^n}\phi_i(x) d\mu(x) \leq   C_{\mu, x_0} \sup_{i\in \mathbb{N}}  \int_{\mathbb{R}^n}\phi_i(x) d\mu(x).
\end{align} 
That is, for any $x_0 \in M_{\mu}$, the sequence $\{\phi_i(x_0)\}_{i\ge 1}$ is  bounded. Thus, by Lemma \ref{Le.roc}, there exist a subsequence $\{\phi_{i_j}\}_{j\ge 1}$ of $\{\phi_i\}_{i\ge 1}$ and an even finite convex function $\phi_0$ on $M_{\mu}$ so that $\phi_{i_j}\to\phi_0 $ on $ M_\mu$, pointwise as $j\to\infty$. The even convex function $\phi_0: M_{\mu}\rightarrow [0, \infty)$ can be extended to a function on $\R^n$ by
\begin{equation}
\label{eq6}
\phi_0(x)=\left\{
\begin{aligned}
&\lim_{\lambda\to1^-}\phi_0(\lambda x),& x\in \overline{ M_\mu }; \\
&\ +\infty,  &x\notin \overline{ M_\mu }.
\end{aligned}
\right.
\end{equation} Note that the existence of $\lim_{\lambda\to1^-}\phi_0(\lambda x)$ is guaranteed by the fact that $\phi_0(\lambda x)$ is increasing with respect to $\lambda\in(0,1)$, an easy consequence following the fact that $\phi_0(x)$ is an even convex function. Consequently, for any  $\lambda\in(0,1)$, by Fatou's lemma, convexity of $\phi_{i_j}$, the fact that $\phi_{i_j}(o)\rightarrow \phi_{0}(o)$, and \eqref{In-4}, one has
\begin{align*}
\int_{\mathbb{R}^n}\phi_0(\lambda x)d\mu
&\le\liminf_{j\to\infty}\int_{\mathbb{R}^n}\phi_{i_j}(\lambda x)d\mu\\
&\le\lambda \liminf_{j\to\infty}\int_{\mathbb{R}^n}\phi_{i_j}(x)d\mu+
(1-\lambda)\phi_0(o)|\mu|<\infty.
\end{align*} 
This, the fact that $\mu((\overline{M_\mu})^c)=0$, and the monotone convergence theorem, immediately imply the first inequality in \eqref{two-terms}. In particular, $\phi_0\in \LP$.

Let $\{x_i\}_{i\in \mathbb{N}}$ be a dense sequence in $M_\mu$. For every $y\in \R^n$, by \eqref{eq6} and the continuity of $\phi_0$ in $M_{\mu}$, we have
$$\phi_0^\ast(y)=\sup_{x\in M_\mu}\big\{\langle x,  y\rangle -\phi_0(x)\big\}=\sup_{i\in \mathbb{N}}\big \{\langle x_i,  y\rangle-\phi_0 (x_i)\big\}.$$  
For $k\ge1$, let 
$$\phi^\ast_{0,k}(y)=\max_{i\in [1, k]\cap \mathbb{N}}\big\{ \langle x_i,  y\rangle -\phi_0(x_i)\big\}.$$
 Clearly, for sufficiently large $k$ (and consequently at least one $x_i$ for $i\leq k$ is not the origin), we have $e^{-\phi_{0,k}^\ast}\in \LC$, and thus  $\VO(e^{-\phi_{0,k}^\ast})<\infty$ due to Lemma \ref{Le.finiteness}. It follows from the monotone convergence theorem that
\begin{align}
\lim_{k\to\infty}\widetilde{V}_\omega(e^{-\phi_{0,k}^\ast})=
\widetilde{V}_\omega(e^{-\phi_0^\ast}).
\end{align} 
Thus, for any $\varepsilon>0$, there exists $k_0\in \mathbb{N}$ (big enough) such that
 $$\widetilde{V}_\omega(e^{-\phi_0 ^\ast})\ge\widetilde{V}_\omega(e^{-\phi_{0,k_0}^\ast})-\varepsilon.$$ 
 Since $\phi_{i_j}\to\phi_0$ pointwise at $\{x_1,\cdots,x_{k_0}\}$, there exists $j_0>0$ such that for all $j>j_0$, we have $\phi^\ast_{i_j}(y) \ge\phi_{0,k_0}^\ast(y)-\varepsilon$.
Hence, for all $j\geq j_0$, we have
\begin{align*}
\widetilde{V}_\omega(e^{-\phi_0 ^\ast})\ge \widetilde{V}_\omega(e^{-\phi_{0,k_0}^\ast})-\varepsilon
\ge e^{-\varepsilon}\widetilde{V}_\omega(e^{-\phi_{i_j}^\ast})-\varepsilon.
\end{align*}
Letting $j\rightarrow \infty$, we have
\begin{equation*}
	\widetilde{V}_\omega(e^{-\phi_0^*})\geq e^{-\varepsilon} \limsup_{j\rightarrow \infty}\widetilde{V}_{\omega}(e^{-\phi_{i_j}^*})-\varepsilon.
\end{equation*}
 This completes the proof since $\varepsilon>0$ is arbitrary.
\end{proof}

Let $t\in(0,\infty)$. Define 
\begin{align}\label{int-t-ball-1}
  b_t= \int_{tB_2^n}\omega(x)dx = \widetilde{V}_{\omega}(\mathbf{1}_{tB_2^n})
\end{align} 
Lemma \ref{Le.finiteness} implies that $b_t$ is finite for every $t$. An application of the monotone convergence theorem then implies 
\begin{equation}
\label{eq 8151}
	\lim_{t\rightarrow 0^+} b_t=0.
\end{equation}

We first observe that the set of functions satisfying the constraints of the optimization problem \eqref{optim.P} is non-empty. Indeed, by \eqref{eq 8151}, there exists $t_0>0$ such that 
 that $0<b_{t_0}\leq a$.
 Take the function
\begin{align}\label{tidle-varphi}
\widetilde{\phi}(x)=\log a-\log b_{t_0}+t_0|x|. 
\end{align} 
Since $\mu$ has a finite first moment, we have
$\widetilde{\phi}\in\LP$. Moreover, 
$$\widetilde{\phi}^\ast(x)=-\log a+\log b_{t_0}+\chi_{t_0B_2^n}.$$
Consequently
$\VO(e^{-\widetilde{\phi}^\ast})= a$. Thus, $\widetilde{\phi}$ satisfies the constraints of the optimization problem \eqref{optim.P}.

\begin{lemma}\label{theorem-with-0} Let $\omega\in\mathscr{G}$, $\mu\in\mathscr{M}$, and $a>0$. Then the optimization problem \eqref{optim.P} admits a solution, say $\phi_0$, such that, $\phi_0\in \LP$ is a lower semi-continuous convex function. 
\end{lemma}

\begin{proof}
 Let $\{\phi_i\}_{i\in\mathbb{N}}$ be a minimizing sequence
to \eqref{optim.P}. Note that  $0\leq \phi_i^{\ast\ast}\le\phi_i$ and $(\phi_i^{\ast\ast})^\ast=\phi_i^{\ast}$. Thus
   $$\int_{\mathbb{R}^n}\phi_i^{\ast\ast}d\mu\le\int_{\mathbb{R}^n}\phi_i d\mu\ \  \mathrm{and} \ \  \widetilde{V}_\omega(e^{-(\phi_i^{\ast\ast})^\ast})=\widetilde{V}_\omega(e^{-\phi_i^\ast}).$$ 
   Hence, we may assume $\phi_i\in\LP$ is lower semi-continuous and convex for any $i\in \mathbb{N}$ (by possibly replacing $\phi_i$ by $\phi_i^{\ast\ast})$. Lemma \ref{Le.existence} now implies the existence of a convex function $\phi_0\in\LP$ such that \eqref{two-terms} holds. By a similar argument as before, one may replace $\phi_0$ by $\phi_0^{\ast\ast}$ and \eqref{two-terms} still holds. Thus, one gets that $\phi_0$ is also lower semi-continuous. Equation \eqref{two-terms} immediately shows that $\phi_0$ is a solution to the optimization problem \eqref{optim.P}.
\end{proof}
 
It is important to note that in order to use the Lagrange equation to conclude that a minimizer for the optimization problem \eqref{optim.P} is a solution to Problem \ref{complete-2}, one needs to be able to perturb the minimizer. Therefore, small perturbations of the minimizer have to satisfy the constraints of the optimization problem \eqref{optim.P} as well. One way to guarantee that is to show \emph{a priori} that the minimizer obtained in Lemma \ref{theorem-with-0} is actually a positive function. We will show in Lemma \ref{5.8} that, with additional assumptions, this is indeed true for sufficiently large $a>0$.

 The behavior of the following functional is essential: for $t\in (0, \infty)$, 
\begin{align}
\mathcal{F}_\omega(t)=\VO\big(e^{-t|x|}\big). \label{def-cal-H} 
\end{align} 
By Lemma \ref{Le.finiteness} with $f=e^{-t|x|}$, we see that $\mathcal{F}_\omega (t)$ is well defined  for all $t\in (0, \infty)$. 

Recall from \eqref{c0} that for each $\mu\in \mathscr{M}$, we have $0<\zeta_{\mu}<\infty$.

\begin{lemma}\label{Le.H-1} Let $\omega\in\mathscr{G}$, $\mu\in\mathscr{M}$, and $a>0$.  Let $\phi\in \LP$ be a lower semi-continuous, convex function such that $\widetilde{V}_\omega(e^{-\phi^\ast})\ge a$ and $\phi(o)=0$. Then,
\begin{align}\label{H-1}
\frac{a}{e}\leq\mathcal{F}_\omega\left(\frac{\zeta_{\mu}}{2\int_{\mathbb{R}^n}\phi \,d\mu+2|\mu|}\right).
\end{align}
\end{lemma}

\begin{proof} Let $K_{\phi}=\{x\in\mathbb{R}^n:\phi(x)\le 1\}$ and define
$r_{\phi}=\min_{v\in S^{n-1}}h_{K_{\phi}}(v)$. Since $\phi$ is $\mu$-integrable, it is finite in a neighborhood of the origin. This, in combination with the convexity of $\phi$ and that $\phi(o)=0$, implies $r_\phi>0$. Since $\widetilde{V}_\omega(e^{-\phi^\ast})>0$, the function $\phi^*$ is not almost everywhere $+\infty$. This and the fact that $\phi^*$ is even and convex show that $\phi^*$ is finite in a neighborhood of the origin. This, the fact that $(\phi^\ast)^\ast=\phi$, and the definition of Legendre transform, show $\phi\rightarrow \infty$ as $|x|\rightarrow \infty$. Thus $0<r_{\phi}<\infty$. (As a matter of fact, $K_\phi\in \mathcal{K}_o^n$ is origin-symmetric.)

Let us first claim that
\begin{align}\label{r-upperbound}
\frac{a}{e}\leq \mathcal{F}_\omega(r_{\phi}).
\end{align} 
Indeed, 
\begin{align*}
\phi^\ast(y)=\sup_{x\in\mathbb{R}^n}\big\{\langle x,y\rangle-\phi(x)\big\}\ge\sup_{x\in r_\phi B_2^n}\big\{\langle x,y\rangle-\phi(x)\big\}\ge r_{\phi}|y|-1
\end{align*}
for any $y\in\mathbb{R}^n$. From this and \eqref{def-cal-H} , one gets
\begin{align}
a\le \widetilde{V}_\omega(e^{-\phi^\ast}) \leq \int_{\R^n} e^{1-r_{\phi}|y|}\omega(y) \,dy=e \mathcal{F}_\omega(r_{\phi}),
\end{align}
which implies \eqref{r-upperbound}.

On the other hand, choose $v_0\in S^{n-1}$ such that $r_{\phi}=h_{K_{\phi}}(v_0)$.  
Note that
if $x\in\mathbb{R}^n$ satisfies $|\langle x, v_0\rangle|> r_{\phi}$, then $x\notin K_{\phi}$. 
For any $x\in\mathbb{R}^n$ with $|\langle x, v_0\rangle|>2r_{\phi}$, by the convexity of $\phi$ and $\phi(o)=0$, one has
\begin{align*}
\frac{2r_{\phi}}{|\langle x, v_0\rangle|}{\phi}(x)
\ge{\phi}\left(\frac{2r_{\phi}}{|\langle x, v_0\rangle|}x+\left(1-\frac{2r_{\phi}}{|\langle x, v_0\rangle|}\right)o\right)
={\phi}\left(\frac{2r_{\phi}}{|\langle x, v_0\rangle|}x\right)\ge 1,
\end{align*}
where
the last inequality is due to $\frac{2r_{\phi}}{|\langle x, v_0\rangle|}x\notin K_{\phi}$. Therefore, since $\phi$ is nonnegative, for all $x\in \R^n$,
\begin{align}\label{in-r-varphi}
2r_{\phi}(\phi(x)+1)\ge|\langle x,v_0\rangle|.
\end{align} 
Integrating both sides with respect to the measure $\mu$, one has
\begin{align}
r_{\phi}\ge
\frac{\int_{\mathbb{R}^n}|\langle x,v_0\rangle|d\mu}{2\int_{\mathbb{R}^n}\phi(x)d\mu+2|\mu|}\ge
\frac{\zeta_{\mu}}{2\int_{\mathbb{R}^n}\phi(x)d\mu+2|\mu|}.
\end{align}
Since $\mathcal{F}_\omega$ is monotone decreasing, this implies
\begin{equation}\label{eq 8153}
	\mathcal{F}_\omega(r_{\phi})\leq \mathcal{F}_\omega\left(\frac{\zeta_{\mu}}{2\int_{\mathbb{R}^n}\phi(x)d\mu+2|\mu|}\right).
\end{equation}
The desired \eqref{H-1} now follows from \eqref{r-upperbound} and \eqref{eq 8153}. 
\end{proof}

The following lemma asserts that under certain conditions on $\mu$, the solution $\phi_0$ found in Lemma \ref{theorem-with-0} must be strictly positive for sufficiently large $a>0$.

\begin{lemma}\label{5.8} Let $\omega\in\mathscr{G}$ and $\mu\in\mathscr{M}$. Assume $\mu$ is such that 
\begin{align}\label{limit-condition}
\liminf_{t\to 0^+ } \Big(\mathcal{F}_\omega(t) e^{-\frac{\zeta_{\mu}}{2|\mu|t}}\Big)=0,
\end{align}
where $\zeta_{\mu}$ is the constant given by \eqref{c0}. There exists $a>0$ sufficiently large such that if $\phi_0\in \LP$ is a convex, lower semi-continuous solution to the optimization problem \eqref{optim.P} found in Lemma \ref{theorem-with-0}, then $\phi_0>0$ and $\widetilde{V}_\omega(e^{-\phi_0^\ast})=a$.
\end{lemma}

\begin{proof}
For each $a>0$, let $\phi_{0,a}\in \LP$ be the convex, lower semi-continuous solution to the optimization problem \eqref{optim.P} found in Lemma \ref{theorem-with-0}.  Recall that $\phi_{0,a}$ is even and therefore achieves its minimum at the origin. We argue by contradiction, and assume for any $a>1$, we have $\phi_{0,a}(o)=0$. Thus, by Lemma \ref{Le.H-1}, we have
\begin{equation}
\label{eq 8154}
	\frac{a}{e} \le \mathcal{F}_\omega\left(\frac{\zeta_{\mu}}{2\int_{\mathbb{R}^n}\phi_{0,a} d\mu+2|\mu|}\right).
\end{equation}

According to \eqref{eq 8151}, there exists $t_0>0$ such that 
 that $0<b_{t_0}\leq 1<a$. As computed before, the function
\begin{align}
\widetilde{\phi}_a(x)=\log a-\log b_{t_0}+t_0|x|,
\end{align} 
satisfies the constraints of the optimization problem \eqref{optim.P}. Since $\phi_{0,a}$ is a minimizer, we have
\begin{equation}
\label{eq 8155}
	\int_{\R^n} \phi_{0,a}d\mu \leq \int_{\R^n}\widetilde{\phi}_ad\mu = (\log a-\log b_{t_0})|\mu| + t_0 \int_{\R^n}|x|d\mu=|\mu|\log a+c_1,
\end{equation}
where $c_1= -\log b_{t_0}|\mu|+t_0\int_{\R^n}|x|d\mu$ is independent of $a$.

Equations \eqref{eq 8154}, \eqref{eq 8155}, and the monotonicity of $\mathcal{F}_\omega$ imply
\begin{equation}
\label{8157}
	\frac{a}{e}\leq \mathcal{F}_\omega\left(\frac{\zeta_{\mu}}{2|\mu|\log a+2c_1+2|\mu|}\right)=\mathcal{F}_\omega\left(\frac{\zeta_{\mu}}{2|\mu|\log a+c_2}\right),
\end{equation}
where $c_2 = 2c_1+2|\mu|$ is independent of $a$. Let $t=t(a)$ be 
$$t=\frac{\zeta_{\mu}}{2|\mu|\log a+c_2}.$$
Note that $\lim_{a\rightarrow \infty} t=0$ and $a=e^{\frac{\zeta_{\mu}-c_2 t}{2|\mu|t}}$. Hence, \eqref{8157} implies
\begin{equation}
	 e^{-1-\frac{c_2 }{2|\mu|}}\leq \mathcal{F}
	_\omega(t)e^{-\frac{\zeta_{\mu}}{2|\mu|t}}.
\end{equation}
Letting $a\rightarrow \infty$ (or, equivalently, $t\rightarrow 0$), we have a contradiction to \eqref{limit-condition}. Hence, there exists $a>0$ such that $\phi_0=\phi_{0,a}>0$.

We now show $\widetilde{V}_\omega(e^{-\phi_0^\ast})=a$. Assume (for sake of contradiction) that $\widetilde{V}_\omega(e^{-\phi_0^\ast})>a$. Choose a constant $\varepsilon>0$ small enough such that $\phi_\varepsilon(x)=\phi_0(x)-\varepsilon>0$  for any $x\in\mathbb{R}^n$. This can be done since $\inf \phi_0=\phi_0(o)>0$. Hence, $\phi_{\varepsilon}\in \LP$ is a lower semi-continuous convex function, and $\phi_{\varepsilon}^*=\phi_0^*+\varepsilon$. This further gives, for $\varepsilon$ small enough, 
\begin{align*}
\widetilde{V}_\omega(e^{-\phi_\varepsilon^\ast})
=e^{-\varepsilon}\widetilde{V}_\omega(e^{-\phi_0^\ast})
\ge a.
\end{align*} 
However, $
\int_{\mathbb{R}^n}\phi_0d\mu>\int_{\mathbb{R}^n}\phi_\varepsilon d\mu$. This contradicts the minimality of $\phi_0$ to the optimization problem (\ref{optim.P}). 
\end{proof}

We are now in the position to state and prove our main theorem.

\begin{theorem}\label{Thm 5.9} Let $\omega\in\mathscr{G}$ be an even function and $\mu\in\mathscr{M}$ be a nonzero even Borel measure on $\R^n$. If \eqref{limit-condition} holds, then there exists $f_0\in \LC$ such that 
\begin{align}\label{sol-const-11-23}
\mu= {\widetilde{C}^{e}_{\omega}(f_0,\cdot)}.
\end{align}
\end{theorem}
\begin{proof} By Lemma \ref{5.8}, there exists $a>0$ such that the solution $\phi_0$ to the optimization problem \eqref{optim.P} found in Lemma \ref{theorem-with-0} is even, convex, lower semi-continuous, $\phi_0>0$, and $\widetilde{V}_\omega(e^{-\phi_0^\ast})=a$.

Let $g:\mathbb{R}^n\to\mathbb{R}$ be an even, compactly supported, continuous function. Define  $\phi_t(x)=\phi_0(x)+tg(x)$ for  $t\in[-t_0,t_0]$ where $t_0\in(0,1)$ is such  that $\phi_t\in \LP$ for all $t\in [-t_0, t_0]$. This is possible because $g$ is continuous and compactly supported and consequently, $|g(x)|\le M$ for all $x\in\mathbb{R}^n$ for some $M>0$. Moreover, for any $t\in[-t_0,t_0]$ and $x, y\in \R^n$, \begin{align}
 |\phi_t(x)-\phi_0(x)| \leq |t|M \ \ \mathrm{and} \ \ 
|\phi_t^\ast(y)-\phi_0^\ast(y)|\le|t|M.  \label{varphi-t^ast}
\end{align}

Define 
$$\widetilde{\phi}_t=\phi_t+\log a-\log \widetilde{V}_\omega(e^{-\phi_t^\ast}).$$ 
Note that since $\widetilde{V}_\omega(e^{-\phi_0^\ast})=a$, if $t_0\in (0,1)$ is sufficiently small, then we have for all $t\in [-t_0, t_0]$, 
\begin{align}
\widetilde{\phi}_t\geq \phi_0-|t|M+\log a-\log \widetilde{V}_\omega(e^{-\phi_0^\ast+|t|M})=\phi_0-2|t|M\geq 0. \label{bound-23-f1}
\end{align}
Thus, $\widetilde{\phi}_t\in \LP$. Moreover,  \begin{align*}
\widetilde{\phi}_t^\ast
=\phi_t^\ast-\log a+\log \widetilde{V}_\omega(e^{-\phi_t^\ast})
\ \ \mathrm{and} \ \  
\widetilde{V}_\omega(e^{-\widetilde{\phi}_t^\ast})=a.
\end{align*} 
In particular, $\widetilde{\phi}_t$ satisfies the constraints of the optimization problem \eqref{optim.P}. It can also be checked, by \eqref{varphi-t^ast}, that for $t\in[-t_0,t_0]$, 
 \begin{align*}
\left|\frac{e^{-\phi_t^\ast}-e^{-\phi_0^\ast}}{t}\right| 
\le e^{-\phi_0^\ast}\cdot
\max\left\{\left|\frac{e^{tM}-1}{t}\right|,\left|\frac{e^{-tM}-1}{t}\right|\right\}.
\end{align*} 
Clearly, the maximum on the right-hand side is uniformly bounded with respect to $t\in[-t_0,t_0]\backslash \{0\}$. Hence, the  dominated convergence theorem can be applied to get
\begin{align}\label{variation,t=0}
\frac{d}{dt}\widetilde{V}_\omega(e^{-\phi_t^\ast})\bigg|_{t=0}
=-\int_{\mathbb{R}^n}\frac{d\phi^\ast_t}{dt}\bigg|_{t=0} e^{-\phi_0^\ast(x)}\omega(x)dx
=\int_{\mathbb{R}^n}g(\nabla\phi_0^\ast(x)) e^{-\phi_0^\ast(x)}\omega(x)dx, 
\end{align} where the last equality follows from Lemma \ref{rotem}. 

Note that $\widetilde{\phi}_0=\phi_0$.
Following a calculation similar to that in \eqref{bound-23-f1}, one gets \begin{align*}
\left|\frac{\widetilde{\phi}_t-\widetilde{\phi}_0}{t}\right|
\le 2M.
\end{align*} Again one can apply the dominated convergence theorem to get 
\begin{equation}
\label{eq 8159}
\begin{aligned}
	\frac{d}{dt}\int_{\mathbb{R}^n} \widetilde{\phi}_t\,d\mu\bigg|_{t=0}
&=\int_{\mathbb{R}^n}\frac{d \widetilde{\phi}_t}{dt} \bigg|_{t=0}d\mu\\
&=\int_{\mathbb{R}^n}g(x)d\mu-
\frac{|\mu|}{\widetilde{V}_\omega(e^{-\phi_0^\ast})}\int_{\mathbb{R}^n}g(\nabla\phi_0^\ast(x)) e^{-\phi_0^\ast(x)}\omega(x)dx,
\end{aligned}
\end{equation} 
where we have used $\phi_t(x)=\phi_0(x)+tg(x)$ and \eqref{variation,t=0} in the last equality. As $\phi_0$ is the solution of the minimization problem (\ref{optim.P}), one must have $$\frac{d}{dt}\int_{\mathbb{R}^n} \widetilde{\phi}_t\,d\mu\bigg|_{t=0}=0$$ and thus, by \eqref{eq 8159}, 
\begin{align*}
\int_{\mathbb{R}^n}g(x)d\mu(x)=\frac{|\mu|}{\widetilde{V}_\omega(e^{-\phi_0^\ast})}\int_{\mathbb{R}^n}g(\nabla\phi_0^\ast(x)) e^{-\phi_0^\ast(x)}\omega(x)dx=
\frac{|\mu|}{\widetilde{V}_\omega(e^{-\phi_0^\ast})}\int_{\mathbb{R}^n}g(x) d\widetilde{C}_\omega^e(e^{-\phi_0^\ast},x)
\end{align*}
for any even compactly supported continuous function $g$ on $\mathbb{R}^n$. As both $\omega$ and $\phi_0^*$ are even, the measure $\widetilde{C}_\omega^e(e^{-\phi_0^\ast},\cdot)$ is an even Borel measure. This ensures that
\begin{align*}
\mu= \frac{|\mu|}{V_{\omega}(e^{-\phi_0^\ast})} \widetilde{C}_\omega^e(e^{-\phi_0^\ast},\cdot).
\end{align*} 
Let $f_0=\frac{|\mu|}{V_{\omega}(e^{-\phi_0^\ast})}e^{-\phi_0^\ast}$. Since $\phi_0$ is finite in a neighborhood of the origin, one can immediately obtain that $f_0\in\LC$ by applying \cite[Theorem 11.8 (c)]{RW1998} and the fact that \eqref{In.CF13} holds for $f_0$ is equivalent to the integrability of $f_0$. From the homogeneity of the dual Orlicz curvature measure, we obtain
$\mu=\widetilde{C}_\omega^e(f_0,\cdot).$ 
This completes the proof.
\end{proof}

\begin{remark}\label{remark 8.29.1}
	If $\omega$ is such that $\mathcal{F}_\omega(t)$ grows sufficiently slow as $t\rightarrow 0^+$, then condition \eqref{limit-condition} is satisfied for any $\mu \in \mathscr{M}$. 
	In particular, if $\int_{\R^n}\omega dx$ is finite, the function $\mathcal{F}_\omega$ is a bounded function and therefore condition \eqref{limit-condition} is trivially satisfied for any $\mu\in \mathscr{M}$. This, in particular, contains the case where $\omega(x)=e^{-\frac{|x|^2}{2}}$ is the Gaussian density.
	
	It is also simple to see that condition \eqref{limit-condition} is trivially satisfied, if $\omega(x)=|x|^{q-n}$ for $q>0$. Indeed, using polar coordinates, 
	\begin{align*}\mathcal{F}(t)&=\int_{\R^n} e^{-t|x|}|x|^{q-n}dx =\int_{\sphere} \int_{0}^\infty e^{-tr}r^{q-1}dr\,du \\ &
=t^{-q} \int_{\sphere}\int_{0}^\infty e^{-s}s^{q-1}\,ds\,du=t^{-q}n V_n(\ball) \Gamma(q),
\end{align*}
where $\Gamma(\cdot)$ is the Gamma function. In other words, $\mathcal{F}(t)$ is of polynomial growth near the origin. Hence, condition \eqref{limit-condition} is trivially satisfied for any $\mu\in \mathscr{M}$. In particular, in this case, Theorem 
\ref{Thm 5.9} recovers those in  \cite{CK15, HLXZ}.
\end{remark}

\vskip 2mm \noindent  {\bf Acknowledgement.}  The
research of NF was supported by  NSFC (No.12001291) and the Fundamental Research Funds for the Central Universities (No.531118010593). The
research of DY was supported by a NSERC grant, Canada.
 The research of YZ was supported, in part, by NSF grant DMS-2132330. 
 The research of ZZ was supported by NSFC (No. 12301071) and the Science and Technology Research Program of Chongqing Municipal
Education Commission (No. KJQN202201339).

\end{document}